\documentclass[10pt]{article}
\usepackage{amsmath}
\usepackage{amssymb}
\usepackage{geometry}
\usepackage{amscd}
\usepackage{xypic}
\usepackage{mathrsfs}
\usepackage{amsthm}
\usepackage{color}
\usepackage{fancyhdr}
\usepackage{graphicx}
\usepackage{pgf,tikz}
\usetikzlibrary{arrows}
\usetikzlibrary{calc,intersections,through,backgrounds}
\usepackage{array}
\usepackage{float}

\newtheorem{lemma}{Lemma}[section]
\newtheorem{theorem}{Theorem}[section]
\newtheorem{definition}{Definition}[section]

\newtheorem{corollary}{Corollary}[section]
\theoremstyle{definition}
\newtheorem{example}{Example}[section]
\newtheorem{remark}{Remark}[section]
\newtheorem{algorithm}{Algorithm}
\begin{document}

\title{Algorithms And Programming On The Minimal Combinations Of  Weights Of Projective Hypersurfaces
}

\author{Dun Liang       
}

\maketitle

\begin{abstract}
This paper designs an alogrithm to compute the minimal combinations of finite sets in Euclidean spaces, and applys the algorithm of study the moment maps and geometric invariant stability of hypersurfaces. The classical example of cubic curves is repeated by the algorithm. Furhtermore the alogrithm works for cubic surfaces. For given affinely indepdent subsets of monomials, the algorithm can output the unique unstable points of the Morse strata if it exists. Also there is a discussion on the affinely dependent sets of monomials.
  
\paragraph{keywords}\quad {minimal combination \and hypersurface \and symplectic reduction \and geometric invariant theory \and numerical criterion of stability}

\end{abstract}
\section{Introduction}
\label{intro}

Geometric invariant theory (GIT for short) was founded by Mumford \cite{MFK} to construct the quotients in algebraic geometry. One of the key ideas is the GIT stability, and \cite{MumStability} gives some elementary methods and examples of analyzing GIT stability. A pragmatic tool to study GIT stability is  the Hilbert-Mumford numerical criterion of stability (see \cite{MFK}, Chapter 2, Section 1). There are many works on this topic. Nowadays there are results of GIT stability on much more complicated objects, like \cite{Stabilityg5} and \cite{IMorrison}.

In the complex analytic setting, the GIT quotients correspond to the symplectic quotients of the preimage of the origin of the moment maps, and this is the theory of symplectic reduction (see \cite{Kirwan} and \cite{MW}). Furthermore, \cite{Ness}, \cite{Kirwan} and \cite{Kempf} showed that the Hesselink's stratification on the unstable set \cite{Hesselink} which is given by the positive values of the numerical criterion function coincides with the Morse stratification of the normsquare function of the moment map. This construction and Atiyah's convexity theorem (see \cite{Atiyah}) give rise the idea of the variation of GIT (see \cite{Dolgachev-Hu}).

In this work we come back to \cite{Kirwan} and \cite{Ness}. The GIT stability of hypersurfaces is studied by the moment map in Section 10 of \cite{Ness}, especially for cubic curves. The concept of minimal combinations is introduced by \cite{Kirwan}, to give all the possible positive values of the numerical criterion function, and the index of the Morse stratification of the normsquare of the moment map at the same time. Recently there is a refinement of the Morse stratification (see \cite{Kirwan1}).

Let $V$ be a complex vector space of dimension $n+1$, and ${\mathbb P}(V)\simeq {\mathbb P}^n({\mathbb C})$ be the projective space whose affine cone is $V$. Choose a basis of $V$, we can consider the action of the complex linear reductive group ${\rm GL}(V)\simeq {\rm GL}(n+1,{\mathbb C})$ on the projective space ${\mathbb P}(V)$ induced by the right matrix multiplication of ${\rm GL}(V)$ on $V$. For any complex reductive group $G$, we say $G$ acts linearly on ${\mathbb P}(V)$ if the action of $G$ on ${\mathbb P}(V)$ is given by a linear representation $\rho:G\rightarrow {\rm GL}(V)$. In general, the quotient space ${\mathbb P}(V)/G$  does not exist in any reasonable sense as an algebraic variety. There are various notions 
of quotients in the category of algebraic varieties. One of the main theorems of geometric invariant theory (see \cite{MFK}) asserts  that there exists a Zariski open subset ${\mathbb P}(V)^{\rm ss}$ of ${\mathbb P}(V)$, so called the set of semi-stable points, such that there is  a good categorical quotient, denoted ${\mathbb P}(V)^{\rm ss}// G$.

According to the Hilbert-Mumford numerical criterion of stability (see \cite{MFK}, Chapter 2, Section 1), there exists a real-valued function $M:{\mathbb P}(V)\rightarrow {\mathbb R}$, such that ${\mathbb P}(V)^{\rm ss}=\{x\in {\mathbb P}(V)\ | \ M(x)\leq 0\}$. The complementary set to the semi-stable points, namely  ${\mathbb P}(V)^{\rm us}= \{x\in {\mathbb P}(V)\ | \ M(x)> 0\}$ is called the set of unstable points.

Given a complex reductive group $G$ acting linearly on ${\mathbb P}(V)$ in the category of  algebraic varieties, we can view this as an analytic action of a complex Lie group on a complex manifold. Then the maximal compact subgroup $K$ of the corresponding complex Lie group $G$ acts symplectically on ${\mathbb P}(V)$. This means the image of $K$ under $\rho$ is contained in the unitary group ${\rm U}(V)$. That is, we may chose a suitable basis  of $V$ such that the 
action of $K$ is via unitary matrices.

If $\frak k$ is the Lie algebra of $K$, then its complexification  $\frak k\otimes _{\mathbb{R}} \mathbb{C}$ is 
isomorphic to $\frak g$, the Lie algebra of $G$.
Let ${\frak g}^\vee ={\rm Hom}_{\mathbb C}({\frak g},{\mathbb C})$. There exists a moment map ${\frak m}:{\mathbb P}(V)\rightarrow {\frak g}^\vee$ with respect to the linear action of $K$ on ${\mathbb P}(V)$. The theory of symplectic reduction shows that there is an isomorphism (see \cite{Kempf}) $${\frak m}^{-1}(0) \simeq {\mathbb P}(V)^{\rm ss}//G.$$ 

For $x\in {\mathbb P}(V)$, let $K.x$ be the $K$-orbit of $x$ and let $\overline{K.x}$ be its Zariski closure in ${\mathbb P}(V)$. Let $0$ be the zero element of ${\frak g}^\vee$, and $\overrightarrow{\sf d}$ be the signed distance defined as in (\ref{signed}). For the numerical criterion function $M(x)$, we have (see \cite{Ness},Lemma 3.1 and Lemma 3.2 or \cite{Dolgachev-Hu}, Theorem 2.1.9) $$M(x)=\overrightarrow{\sf d}(0,{\frak m}(\overline{K.x})).$$

According to \cite{Kirwan} and \cite{Ness}, it is very important to study the critical points of the function $\|{\frak m}\|^2$ and the function ${\frak m}_T$ which is the restriction of the moment map on a given maximal torus $T$ of $K$. For example, an unstable point in ${\mathbb P}(V)$ must be a critical point of $\|{\frak m}\|^2$. Let ${\mathscr M}$ be the set of $T$-fixed points, then ${\mathscr M}$ is a finite set. Let
 ${\mathbb A}={\frak m}( \mathscr M )$, then ${\mathbb A}\subset {\frak m}_T({\mathbb P}(V))\subset {\frak t}^\vee$ where ${\frak t}$ is the Lie algebra of $T$. 

The concept of minimal combinations (see Definition \ref{mc}) was introduced by \cite{Kirwan}. For the set 
${\mathbb A  } = {\frak m}( \mathscr M )$ in the Euclidean space 
${\frak t}^\vee$
we define the set of minimal combinations, denoted ${\mathbb A}^{\cal B}$, by the following condition:
 $\beta \in {\frak t}^\vee$  belongs to  ${\mathbb A}^{\cal B}$ if and only if there exists a subset $S\subset {\mathbb A}$ such that $\beta$ is the nearest point from $0$ to the convex polytope generated by $S$.

 Choose a Weyl chamber ${\frak t}_+\subset {\frak t}$, and let ${\mathbb A}^{\cal B}_+={\mathbb A}^{\cal B}\cap {\frak t}^\vee_+$.
 For any  $\beta \in {\mathbb A}^{\cal B}_+$, define $Z_\beta$ to be the affine subspace generated by $\alpha \in {\mathscr M}$ such that ${\frak m}(\alpha)\perp \beta$ in ${\frak t}^\vee \simeq {\mathbb R}^{n+1}$. Then define 
  $C_\beta = K.(Z_\beta \cap {\frak m}^{-1}(\beta))$. It is shown in \cite{Kirwan} that 
  the critical set of $\|{\frak m}_T\|^2$ is the disjoint union of the $C_{\beta}$ for all $\beta \in {\mathbb A}^{\cal B}_+$.

The concept of the minimal combinations of a finite set in a Euclidean space is independent to the geometric background 
from which 
it originates. In this work we construct the general algorithms to compute the minimal combinations. We make an improvement of the description of the set ${\mathbb A}^{\cal B}$, which is summarized as the following theorem.
\begin{theorem} Let ${\mathbb A}$ be a finite set in a Euclidean space. Let ${\mathbb A}^{\cal B}$ be the set of minimal combinations of ${\mathbb A}$. If $\beta \in  {\mathbb A}^{\cal B}$ and $\beta \notin {\mathbb A}$, then there exists an affinely independent subset $S\subset {\mathbb A}$, such that $\sharp(S)\geq 2$, the point $\beta $ is the nearest point from $0$ to the convex polytope ${\cal C}(S)$ generated by $S$, and $\beta$ is contained in the relative interior of ${\cal C}(S)$.
\end{theorem}

Based on this theorem, we construct an algorithm for computing ${\mathbb A}^{\cal B}$, and realize it in SAGE. Then we apply it to the GIT stability and moment map problem for projective hypersurfaces. This time $G={\rm SL}(n,{\mathbb C})$ and $V={\rm Sym}^d(({\mathbb C}^\vee)^n)$ is the $\binom{n+d-1}{d}$-dimensional ${\mathbb C}$-vector space of $n$ variables homogeneous polynomials of degree $d$. The moment map is explicit for this situation (see \cite{Ness}, Lemma 10.1). 

This paper is orgainzed by the following.
In Section \ref{AOMC} we introduce the concept of minimal combinations, independent with the geometric background which it comes from. We prove our main result Theorem \ref{ak}, and design Algorithm \ref{Algorithm1} to compute the minimal combinations.
In Section \ref{MMMCWH} we illustrate the geometric backgrounds of the symplectic reduction theory of hypersurfaces. In order to compute some concrete results, we prove Corollary \ref{Fbeta}.
 In Section \ref{SOCC} and Section \ref{SOCS}, we program Algorithm \ref{Algorithm1} and Corollary \ref{Fbeta} for the problem of GIT stability and the moment maps of cubic curves and cubic surfaces by SAGE. In Section \ref{EAD} we discuss an example for what may happen if the the set of monomials is affinely dependent. We affiliate the SAGE notebook in the end of the paper.
 
\section*{Acknowledgement}
The author would like to thank Dr Jingyue Chen in Capital Normal University for her help of the writing and drawing jobs of the paper, she and Dr Jianke Chen in Chinese Media University joined the discussion of this work. Thank Professor Jerome Hoffman in LSU for his helps on the writing of the paper. Also  Professor Yi Hu in University of Arizona for his nice course on this topic in Tsinghua University.

\section{Algorithms on Minimal Combinations}\label{AOMC}

Let $S$ be a finite set in the $n$-dimensional Euclidean space ${\mathbb R}^n$. Later on we will only consider the case when all the points of $S$ have integer coordinates, in this section we do not have this assumption. Let ${\cal C}(S)$ be the convex polytope generated by $S$. Let $O=(0,\ldots ,0)$ be the origin of ${\mathbb R}^n$. By \cite{GTM168}, Lemma 3.1, there exists a unique point $w_S$, such that 
\begin{equation}\label{ws}\|w_S\|= \inf_{x\in {\cal C}(S)}\|x\|.\end{equation}
\begin{definition}\label{tau} Let 
$${\mathscr S}=\{\, S\subset {\mathbb R} \,  | \, S \mbox{\rm\ is finite}, S\neq \varnothing \mbox{{\rm\ and\ }} O\notin S\},$$
by (\ref{ws}), we can define the {\bf map of the shortest point} as 
\begin{equation}\begin{array}{cccc}
\tau : & {\mathscr S} & \longrightarrow & {\mathbb R}^n\\ 
        & S & \longmapsto & w_S \end{array}
\end{equation}
where $w_S$ is defined as (\ref{ws}).
\end{definition}
Definition \ref{tau} is well defined because of the existence and uniqueness of $w_S$. The point $\tau(S)$ is the shortest point of ${\cal C}(S)$. Obviously, the function
$\|\tau \|:$ ${\mathscr S}\rightarrow {\mathbb R}$, $S\mapsto \|\tau(S)\|$ is a decreasing function with respect to the partial order ``$\subset$" on ${\mathscr S}$. That is, if $S_1\subset S_2$, then 
\begin{equation}\label{decreasing}
\|\tau(S_1)\|\geq \|\tau(S_2)\|.
\end{equation}

\begin{definition}[see \cite{Kirwan}]\label{mc} Let $A\in {\mathscr S}$ and let ${\cal P}(A)=2^A\backslash\{\varnothing \}$ be the set of non-empty subsets of $A$. The set 
\begin{equation}
A^{\cal B}:=\tau({\cal P}(A))=\{\tau(S) \ | \  S\subset A, S\neq \varnothing\}\subset {\mathbb R}^n
\end{equation}
is called the {\bf minimal combination} of $A$, elements of $A^{\cal B}$ are called the {\bf minimal combinations} of $A$.
\end{definition}
\begin{remark} The Definition \ref{mc} is well defined. In fact, since $A$ is a finite set, any non-empty subset $S\in {\cal P}(A)$ is also finite. Furthermore, we have ${\cal C}(S)\subset {\cal C}(A)$ because $S\subset A$, thus $O\notin {\cal C}(A)$ implies that $O\notin {\cal C}(S)$, so $S\in {\mathscr S}$. The set $A^{\cal B}$ is the set of the nearest points from the origin $O$ to the convex sets which are generated by the non-empty subsets of $A$.
\end{remark}

We will give an algorithm to compute $A^{\cal B}$, but before that, let us give a better description of $A^{\cal B}$.

\begin{lemma}\label{Aw} Let $A\in {\mathscr S}$. Denote
$$I(A)= \{ \, S\in {\cal P}(A) \ | \  S\ \mbox{\rm is an affinely independent set}\, \}$$
and let 
$$A^{\cal W} = \{ \, \tau(S) \ | \  S\in I(A)\, \},$$
then $A^{\cal B}=A^{\cal W}$.
\end{lemma}

\begin{proof}\ The ``$\supset$" part is obvious. Let us proof the ``$\subset$" part. Let $x\in A^{\cal B}$. By the Definition \ref{mc} of $A^{\cal B}$, there exists $S\in {\cal P}(A)$, such that $x=\tau (S)\in {\cal C}(S)$. Suppose $S=\{x_1,\ldots , x_s\}$. Let $L(S)$ be the linear subspace of ${\mathbb R}^n$ spanned by the vectors $\{x_1-x_0, x_2-x_0, \ldots , x_s-x_0\}$. Then there exists a subset $S'=\{x_1', \ldots , x_{s'}'\} \subset S$, such that $\{x_1'-x_0, \ldots , x_{s'}'-x_0\}$ forms a basis of $L(S)$. Thus $S'$ is affinely independent. Let 
$${\rm Aff}(S)=\left\{ \sum_{i=1}^s \lambda_ix_i \, \right| \left. \sum_{i=1}^s \lambda_i=1, \lambda_1, \ldots , \lambda_s \in {\mathbb R} \right\}$$
be the affine cone of $S$, we have ${\rm Aff}(S')={\rm Aff}(S)$.

Since $x\in {\cal C}(S)\subset {\rm Aff}(S)={\rm Aff}(S')$, we have $x=\lambda_0x_0+\sum_{i=1}^{s'} \lambda_ix_i'$ and $\lambda_0+\lambda_1+\cdots + \lambda_{s'}=1$. But $x\in {\cal C}(S)$, so $\lambda_0, \ldots , \lambda_{s'}\geq 0$, thus $x\in {\cal C}(S')$. Note that $\|\tau\|$ is decreasing with respect to ``$\subset$", we have $\|\tau(S')\|\geq \|\tau(S)\|$ because $S'\subset S$.

On the other hand, we have ${\cal C}(S')\subset {\cal C}(S)$ because $S' \subset S$. Then $\|y\|\leq \|\tau(S')\|$ for all $y\in {\cal C}(S')$, in particular, we have $\|x\|\leq \|\tau(S')\|$. But $\tau(S')$ is the unique shortest point in ${\cal C}(S')$, so we have $x=\tau(S)$. i.e. $x\in A^{\cal W}$. \quad $\blacksquare$
\end{proof}

Let $\sharp(S)$ be the cardinality of $S$. In Lemma \ref{Aw}, we have 
$$I(A)\subset \{\, S\subset {\cal P}(A) \ | \ \sharp(S)\leq n
+1\}$$ because any $n+2$ points in ${\mathbb R^n}$ are affinely dependent. Thus
\begin{equation}\label{Aw1}
A^{\cal W} = \{ \, \tau(S) \ | \  S\ \mbox{is affinely independent and \ }  \sharp(S) \leq n+1 \}.
\end{equation}

Let $S=\{ x_1, \ldots , x_s\}\in {\mathscr S}$. For $i = 1, \ldots , s$, regard the points $x_1, \ldots , x_s$ of $S$ as column vectors in ${\mathbb R}^n$. Let $B_i$ be the matrix 
$$(x_j-x_i \ | \  j=1,2,\ldots , s,\, j\neq i)$$
whose columns are $x_j-x_i$ for $i=1,\ldots s$. Let 
$$P_i=B_i(B_i^TB_i)^{-1}B_i^T,$$ where $B_i^T$ denotes the transpose matrix of $B_i$.

\begin{lemma}[See \cite{Huang}] Notations as above. For any $x\in {\rm Aff}(S)$, let $I$ be the $n \times n$ identity matrix, the point
\begin{equation}\label{x*}
x^*=(I-P_i)x \in {\rm Aff}(S)
\end{equation}
is independent to the choice of $x$ and the index $i$, and $x^*$
is the unique nearest point from the origin $O$ to the affine subspace ${\rm Aff}(S)$.
\end{lemma}
\begin{proof} Let $y=0$ in the proof of the lemma in \cite{Huang}.\quad $\blacksquare$
\end{proof}
\begin{definition} The map
\begin{equation}
\begin{array}{cccc}
\sigma: & {\mathscr S} & \longrightarrow & {\mathbb R}^n \\
        & S & \longmapsto & x^*

\end{array}
\end{equation}
where $x^*$ is defined as in (\ref{x*}) is called the {\bf minimal square map}. The point $\sigma(S)$ is called the {\bf minimal square} of $S$.
\end{definition}

Let ``$\langle \, , \, \rangle$" be the inner product in ${\mathbb R}^n$. For any $x\in {\rm Aff}(S)$, we have $\langle \sigma(S),x\rangle=0$, or we say $\sigma(S)\perp {\rm Aff}(S)$.  

\begin{lemma}\label{ms}Let $S=\{ x_1,\ldots , x_s\}\in {\mathscr S}$ and $S$ be affinely independent.  Suppose $\tau(S)=\sum_{i=1}^s \lambda_ix_i$ such that $\lambda_i>0$ for all $i=1,\ldots ,s$. Then $\sigma(S)\in {\cal C}(S)$.
\end{lemma}

\begin{proof} We prove this lemma by contradiction. Assume $\sigma(S)\notin {\cal C}(S)$, we will show that there exists $i\in \{1,2,\ldots ,s\}$, such that $\lambda_i=0$.

Since ${\cal C}(S) \subset {\rm Aff}(S)$, we have $\|\sigma(S)\|\leq \|\tau(S)\|$. If $\|\sigma(S)\|=\|\tau(S)\|$, we have $\tau(S)\in {\rm Aff}(S)$ because $\tau(S)\in {\cal C}(S)$ and ${\cal C}(S)\subset {\rm Aff}(S)$. By the uniqueness of $\sigma(S)$ in ${\rm Aff}(S)$, we have $\sigma(S)=\tau(S)$. Thus if $\sigma(S)\notin {\cal C}(S)$, we must have $\|\sigma(S)\|<\|\tau(S)\|$.

Let $\sigma(s)=\sum_{i=1}^s \nu_ix_i$ and $\sum_{i=1}^s\nu_i=1$.  We assumed that $\sigma(S)\notin {\cal C}(S)$, thus there exists $i\in \{1, \ldots , s\}$, such that $\nu_i<0$.

Let $${\cal C}(S)^\circ = \left\{ \, \sum_{i=1}^s \lambda_ix_i \, \right| \left. \, 
\sum_{i=1}^s \lambda_i=1, \lambda_1, \ldots , \lambda_s>0\, \right \}$$ be the relative interior of the convex polytope ${\cal C}(S)$. By assumption we have $\tau(S)\in {\cal C}(S)^\circ$ and $\sigma(S)\notin {\cal C}(S)$. The set $S$ is affinely independent, so 
$$\dim {\rm Aff}(S)= \dim {\cal C}(S) = s-1.$$
The polytope ${\cal C}(S)$ is homeomorphic to the $s-1$ dimensional unit ball $B^{s-1}$. Topologically, the affine space ${\rm Aff}(S)$ is homeomorphic to the $s-1$ dimensional Euclidean space ${\mathbb R}^{s-1}$. The polytope ${\cal C}(S)$ is embedded in ${\rm Aff}(S)$ in the unit ball. By point set topology, any path connecting $\tau(S)$ and $\sigma(S)$ will have a non-empty intersection with the boundary $$\partial {\cal C}(S)=\left\{ \, \sum_{i=1}^s \lambda_ix_i \, \right| \left. \, 
\sum_{i=1}^s \lambda_i=1, \exists i \in \{1,\ldots s\} \mbox{\ such that \ } \lambda_i=0  \right \}.$$ In particular, there exists a point $P$ on the boundary $\partial {\cal C}(S)$ that lies on the segment $\overline{\sigma(S)\, \tau(S)}$  connecting $\sigma(S)$ and $\tau(S)$. Since $\overline{\sigma(S)\, \tau(S)}\subset {\rm Aff}(S)$ and $\sigma(S)\perp {\rm Aff}(S)$, we have $\sigma(S)\perp \overline{\sigma(S)\, \tau(S)}$. The point $P$ is on the boundary $\partial {\cal C}(S)$, and $\tau(S)\in {\cal C}(S)^\circ$, $\sigma(S)\notin {\cal C}(S)$, so $P\neq \sigma(S)$ and $P\neq \tau(S)$ (See {\bf Fig}. \ref{fig3}). 

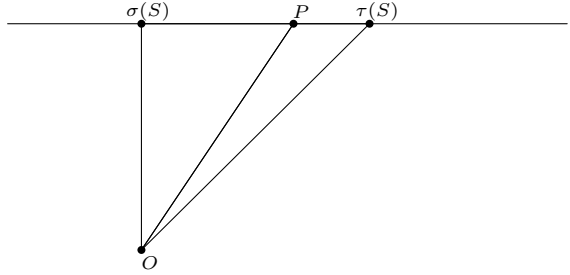
\begin{figure}[H]
\begin{center}
\setlength{\unitlength}{1mm}

%\documentclass[10pt]{article}
%\usepackage{pgf,tikz}
%\usetikzlibrary{arrows}
%\pagestyle{empty}
%\begin{document}
%\definecolor{zzttqq}{rgb}{0.6,0.2,0}
%\definecolor{qqqqff}{rgb}{0,0,1}
%\definecolor{xdxdff}{rgb}{0.49,0.49,1}
\begin{tikzpicture}[line cap=round,line join=round,>=triangle 45,x=1.0cm,y=1.0cm]
\clip(-1.76,-1.39) rectangle (5.6,4.24);
%\fill[color=zzttqq,fill=zzttqq,fill opacity=0.1] (0,0) -- (0,3) -- (2,3) -- cycle;
%\fill[color=zzttqq,fill=zzttqq,fill opacity=0.1] (0,0) -- (2,3) -- (3,3) -- cycle;
\draw [domain=-1.76:5.6] plot(\x,{(--9-0*\x)/3});
\draw  (0,0)-- (0,3);
\draw  (0,3)-- (2,3);
\draw  (2,3)-- (0,0);
\draw  (0,0)-- (2,3);
\draw (2,3)-- (3,3);
\draw (3,3)-- (0,0);
\begin{scriptsize}
\fill  (0,0) circle (1.5pt);
\draw (0.11,-0.17) node {$O$};
\fill  (0,3) circle (1.5pt);
\draw (0.09,3.16) node {$\sigma(S)$};
\fill (2,3) circle (1.5pt);
\draw (2.1,3.16) node {$P$};
\fill  (3,3) circle (1.5pt);
\draw (3.1,3.16) node {$\tau(S)$};
\end{scriptsize}
\end{tikzpicture}
%\end{document}
\end{center}
\caption{The relative positions of $\sigma(S)$, $\tau(S)$ and $P$\label{fig3}}

\end{figure}

Obviously $\|\overline{OP}\|<\|\tau(S)\|$, but this contradicts with the definition of $\tau(S)$ that it will be the shortest point in ${\cal C}(S)$. Finally, there exists $i\in \{1,\ldots , s\}$, such that $\lambda_i=0$. \quad $\blacksquare$

\end{proof}

Let $k$ be any positive integer. Define
$$A_k^{\cal W}= \{\, \tau (S)\ | \ S \mbox{\ is affinely independent and \ }\sharp(S)=k\, \}.$$
From (\ref{Aw1}) we have 
\begin{equation}\label{Aw2}
A^{\cal W}=\bigcup_{k=1}^{n+1} A_k^{\cal W}.
\end{equation}

\begin{theorem}\label{ak} Let $A\in {\mathscr S}$ and let $S=\{ x_1, \ldots ,x_s\}$ be a non-empty affinely independent subset of $A$. Suppose $\sigma(S)=\sum_{i=1}^s\nu_ix_i$. Then
\begin{enumerate}
\item if $\nu_i>0$ for all $i\in \{1,2,\ldots ,s\}$, then $\tau(S)=\sigma(S)$.
\item otherwise, if there exists $i\in \{1,2,\ldots ,s\}$ such that $\nu_i\leq 0$, then there exists $1\leq k \leq \sharp(S)$, such that $\tau(S)\in A_k^{\cal W}$.
\end{enumerate}
\end{theorem}

\begin{proof} \quad {\it Proof of 1.}: If $\nu_i>0$ for all $i=1,2,\ldots ,s$, then $\sigma(S)\in {\cal C}(S)$. By Lemma \ref{ms}, we have $\sigma(S)=\tau(S)$.

{\it Proof of 2.}: If for some $i\in \{1,2,\ldots ,\}$ we have $\nu_i\leq 0$, then there are two possibilities.

Either $\nu_i\geq 0$ for all $i=1,2,\ldots , s$. Then we still have $\sigma(S)\in {\cal C}(S)$. By Lemma \ref{ms},  $\tau(S)=\sigma(S)=\sum_{i=1}^s\nu_ix_i$, and this is the unique convex combination of $\tau(S)$ with respect to $S$, so $\nu_i=0$ for some $i\in \{1,2,\ldots ,s\}$.

Otherwise $\nu_i<0$ for some $i\in \{1,2,\ldots ,s\}$. This time $\sigma(S)\notin {\cal C}(S)$. By the contradiction of Lemma \ref{ms}, let $\tau(S)=\sum_{i=1}^s \lambda_ix_i$, there exists $i\in \{1,2,\ldots ,s\}$, such that $\lambda_i=0$.

No matter which case happens, if  $\tau(S)=\sum_{i=1}^s \lambda_ix_i$,  there exists $i\in \{1,2,\ldots ,s\}$, such that $\lambda_i=0$. So if we take $S'=\{x_i \ | \ \lambda_i\neq 0\}$, then $S'\subsetneq S$ and $\sharp(S')<\sharp(S)$.

As a subset of an affinely independent set $S$, the set $S'$ is also affinely independent. But $\tau(S)\neq 0$ implies that $S'\neq \varnothing$. Thus $S'\in {\mathscr S}$. Since $S'\subset S$ we have $\|\tau(S')\|\geq \|\tau(S)\|$. By the definition of $\tau(S')$, in the polytope ${\cal C}(S')$, since $\tau(S)\in {\cal C}(S')$ we have $\|\tau(S')\|\leq \|\tau(S)\|$. Thus $\tau(S')=\tau(S)$. Let $k=\sharp(S')$, we have $1\leq k < \sharp(S)$ and $\tau(S)\in A_k^{\cal W}$. \quad $\blacksquare$
\end{proof}

Next, let 
$$A_k^{\cal B}=\{\, \tau(S) \ | \ S\in A_k^{\cal W}, \sigma(S) \mbox{\ satisfies (1) in Theorem \ref{ak}}\}.$$
By Theorem \ref{ak}, we have
\begin{equation}\label{akb}
A^{\cal B}=A^{\cal W}=\bigcup_{k=1}^{n+1}A_k^{\cal W}=\bigcup_{k=1}^{n+1}A_k^{\cal B}.
\end{equation}

Let $S=\{x_1,\ldots , x_{n+1}\}\in {\mathscr S}$ be an affinely independent set such that $\sharp(S)=n+1$. Then ${\cal C}(S)$ is homeomorphic to the unit ball in ${\mathbb R}^n$. By the same argument of the proof of Lemma \ref{ms}, we have $\|\tau(S)\|\in \partial {\cal C}(S)$. That is, if $\tau(S)=\sum_{i=1}^{n+1}\lambda_ix_i$, then at least one of the $\lambda_i$'s is zero for $i=1,\ldots n+1$. This implies that $A_{n+1}^{\cal B}=\varnothing$. Finally we have
\begin{equation}\label{akb1}
A^{\cal B}=\bigcup_{k=1}^{n}A_k^{\cal B}
\end{equation} 

In sum, we design an algorithm to compute $A_k^{\cal B}$ for $k=1,\ldots n$ according to (\ref{akb1}).

\begin{algorithm}\label{Algorithm1}
\begin{flushleft}
{\sf \

Input:

\quad \quad     \quad  A finite set $A$ in ${\mathbb R}^n$, the points of $A$ are represented as column vectors $x=(x_1,\ldots ,x_n)^T$. 

$k=1$
$$A_1^{\cal B}=A$$
For $k=2,\ldots ,n$\, :
\begin{align*}
& D_k=\{\, S\in {\cal P}(A) \ | \ \sharp(S)=k\, \}; \\
& A_k^{\cal B}=\varnothing
\end{align*}
Define a total order ``$\prec$" on $D_k$, and write $D_k$ as 
$$D_k=\{S_1,\ldots , S_q\}$$ where $q=\sharp(D_k)=\begin{pmatrix} \sharp(A) \\ k
\end{pmatrix}
$.

\qquad \qquad    For $i=1,\ldots ,q$,let $S_i=\{x_i^{(1)}, \ldots , x_i^{(s)}\}$, take the matrix 
$$B_i=(x_i^{(j)}-x_i^{(1)}\ | \ j=2,\ldots ,s)$$
\qquad \qquad \qquad If\quad  $\det (B_i)=0$

\qquad \qquad \qquad \qquad then \quad $A_k^{\cal B}=A_k^{\cal B},$

\qquad \qquad $i=i+1;$

\qquad \qquad \qquad Else 
$$
x_i^*=\big(\, I-B_i(B_i^TB_i)^{-1}B_i^T\, \big)\, x_i^{(1)}
$$
\qquad \qquad \qquad \quad expand $x^*_i$ with respect to $S_i$ as 
$$ x_i^*=\sum_{j=1}^s \lambda_jx_i^{(j)}$$

\qquad \qquad \qquad \qquad \qquad  If $\forall j=1,\ldots ,s$, $\lambda_j>0$

\qquad \qquad \qquad \qquad \qquad \qquad then\qquad  $A_k^{\cal B}=A_k^{\cal B}\cup \{x_i^*\},$

\qquad \qquad $i=i+1;$

\qquad \qquad \qquad \qquad \qquad  Else \qquad $A_k^{\cal B}=A_k^{\cal B}$

\qquad \qquad $i=i+1;$

Return $A_k^{\cal B}$

For $k=1,\ldots ,n$, 

\qquad Output $A_k^{\cal B}$
}
\end{flushleft}
\end{algorithm}

We will give examples of this algorithm for the geometric invariant theory stability of projective hypersurfaces later in this work.

\section{Moment Maps and Minimal Combinations of   Weights of Hypersurfaces}\label{MMMCWH}

This section is according to \cite{MFK}, \cite{Kirwan} and \cite{Ness}.

First we introduce some notations. Let $R_n={\mathbb C}[\,x_1,\ldots ,x_n\,]$ be the ring of complex polynomials of $n$ variables $x_1,\ldots ,x_n$. Let ${\mathbb Z}^n_+=\{(i_1,\ldots ,i_n)\in {\mathbb Z}^n \, | \, i_l\geq 0,\, l=1,\ldots ,n\}$. For any $\alpha = (i_1,\ldots ,i_n)\in {\mathbb Z}^n_+$, let ${\bf x}^\alpha=x_1^{i_1}\ldots x_n^{i^n}$ be the monomial of ${\bf x}=(x_1,\ldots ,x_n)$. Then any element in $R_n$ could be written as
$$f=\sum_{(i_1,\ldots ,i_n)\in {\mathbb Z}^n_+}c_{i_1,\ldots ,i_n}x_1^{i_1}\ldots x_n^{i^n}=\sum_{\alpha \in {\mathbb Z}^n_+}c_{\alpha}{\bf x}^{\alpha}$$
where $c_{i_1,\ldots ,i_n}=c_{\alpha}\in {\mathbb C}$ for $\alpha \in {\mathbb Z}^n_+$ and $c_{\alpha}=0$ for all but finitely many $\alpha \in {\mathbb Z}^n_+$. Let $|\alpha|=i_0+\cdots +i_n$ for $\alpha = (i_1,\ldots ,i_n)$. Let $R_n^d$ be the set of homogeneous polynomials of degree $d$ in $R_n$. Let ${\mathscr W}_d=\{\alpha\in {\mathbb Z}^n_+ \, | \, |\alpha|=d\}$, then an element in $R_n^d$ is of the form
\begin{equation}\label{rnd}
f=f(x_1,\ldots ,x_n)=\sum_{i_1+\cdots +i_n=d}c_{i_1,\ldots ,i_n}x_1^{i_1}\ldots x_n^{i_n}=f({\bf x})=\sum_{\alpha \in {\mathscr W}_d}c_{\alpha}{\bf x}^{\alpha}
\end{equation}
where $c_{i_1,\ldots ,i_n}=c_{\alpha}\in {\mathbb C}$ for $\alpha = (i_1,\ldots ,i_n)\in {\mathscr W}_d$. 

Let ${\mathscr M}_d=\{\,{\bf x}^{\alpha} \, | \, \alpha\in {\mathscr W}_d \, \}=\{\, x_1^{i_1}\ldots x_n^{i_n} \, | \, i_1+\cdots i_n=d\, \}$ be the set of monomials of degree $d$, then $R_n^d$ is the complex vector space of dimension $\binom{n+d-1}{d}$ which is spanned by the basis ${\mathscr M}_d$.
Let ${\bf P}_n^d={\mathbb P}(R_n^d)$ be the projective space of dimension $\binom{n+d-1}{d}-1$ whose affine cone is $R_n^d\backslash \{0\}$, for $f\in R_n^d\backslash \{0\}$, we also denote the image of $f\in {\bf P}_n^d $ by $f$, and say $f\in {\bf P}_n^d$.

Let $M_n({\mathbb C})$ be the ring of $n\times n$ complex square matrices. Denote $A=(a_{ij})_{n\times n}\in M_n(\mathbb C)$ as the matrix whose $(i,j)$-th entry is $a_{ij}\in {\mathbb C}$ for $i,j=1,\ldots ,n$. Consider the {\bf right} action of the general linear group ${\rm GL}(n,{\mathbb C})=\{\, A\in M_{n\times n}({\mathbb C})\, | \, \det A\neq 0 \, \}$ on ${\bf P}_n^d$. That is, if we substitute the linear transformation ${\bf x}\mapsto {\bf x}\, A$, or equivalently
$$x_i\longmapsto \sum_{i=1}^n x_ia_{ij} \quad \mbox{for} \quad i=1,\ldots ,n$$
into (\ref{rnd}), and let 
\begin{equation}\label{Af}\begin{split}
&A.f=f(\,\sum_{i=1}^n x_1a_{1j},\ldots  \sum_{i=1}^n x_na_{nj}\,)
\\ &=\sum_{i_1+\cdots +i_n=d}c_{i_1\ldots i_n}\left(\sum_{i=1}^n x_1a_{1j}\right)^{i_1}\cdots \left(\sum_{i=1}^n x_na_{nj}\right)^{i_n},\end{split}
\end{equation}
then the parenthesizes on the right hand side of (\ref{Af}) are linear transformations. If we expand them, the total degree will not change. Thus we have $A.f\in {\bf P}_n^d$, and $f\mapsto A.f, \forall A\in {\rm GL}(n,{\mathbb C})$ defines a ${\rm GL}(n,{\mathbb C})$-action on ${\bf P}_n^d$, we call it the right action. 

With this action, we can define the right action of any subgroup of ${\rm GL}(n, {\mathbb C})$ on ${\bf P}_n^d$. In particular, we consider right actions of the special linear group ${\rm SL}(n,{\mathbb C})=\{A\in {\rm GL}(n, {\mathbb C}) \, | \, \det A=1\}$ and the special unitary group ${\rm SU}(n)=\{A\in {\rm SL}(n, {\mathbb C}) \ | \ A^\dagger A=I\}$ on ${\bf P}_n^d$. Here $I$ is the $n\times n$ identity matrix, and if $A=(a_{ij})_{n\times n}\in M_n({\mathbb C})$, then $A^\dagger:=(\overline{a_{ji}})_{n\times n}\in {\mathbb C}$ is the conjugate transpose matrix of $A$, so ${\rm SU}(n)$ is the group of unitary matrices of determinant 1. For these linear groups, the corresponding Lie algebras are ${\frak g}{\frak l}(n,{\mathbb C})=M_{n\times n}({\mathbb C})$, ${\frak s}{\frak l}(n,{\mathbb C})=\{A\in{\frak g}{\frak l}(n,{\mathbb C}) \, | \, {\rm tr}(A)=0 \}$ where ${\rm tr}(A)$ is the trace of the matrix $A$, and ${\frak s}{\frak u}(n)=\{A\in{\frak s}{\frak l}(n,{\mathbb C}) \, | \, A^\dagger + A=0 \}$ , the set of skew-hermitian matrices of trace 0.
\begin{lemma}[See \cite{Ness}]\label{innerproduct} Define an inner product ``$\langle \, , \, \rangle$" on $R_n^d$ as:
\begin{enumerate}
\item The set of monomials ${\mathscr M}_d$ forms an orthogonal basis of $R_n^d$, so $\langle {\bf x}^\alpha, {\bf x}^\beta \rangle=0$ if $\alpha \neq \beta, \alpha,\beta \in {\mathscr W}_d$.
\item Let $\alpha=(i_1,\ldots i_n)\in {\mathscr W}_d$, then
$$\langle {\bf x}^\alpha, {\bf x}^\alpha \rangle= \|x_1^{i_1}\cdots x_n^{i_n}\|^2=\frac{i_1!\cdots i_n!}{d!}.$$ 
\end{enumerate}
Then this inner product is ${\rm SL}(n,{\mathbb C})$(also ${\rm SU}(n)$)-invariant. It induces an ${\rm SL}(n,{\mathbb C})$(also ${\rm SU}(n,{\mathbb C})$)-invariant inner product on ${\bf P}_n^d$.
\end{lemma}

\begin{theorem}[See \cite{Ness}]\label{moment} Let $f=\sum_{\alpha \in {\mathscr W}_d} c_\alpha {\bf x}^\alpha \in {\bf P}_n^d$, define the {\bf Hessian}  $H(f)$ of $f$ as the hermitian matrix whose $(i,j)$-th entry is 
\begin{equation}\label{Hesse}
H(f)_{ij}=\frac{1}{d \|f\|^2} \left\langle \frac{\partial f}{\partial x_i}, \frac{\partial f}{\partial x_i} \right\rangle.
\end{equation}
Then the map\begin{equation}
\begin{array}{cccc}
{\frak m}: & {\bf P}_n^d & \longrightarrow & \sqrt{-1}{\frak s}{\frak u}(n ,\mathbb C) \\
& f & \longmapsto & \displaystyle{2\left(H(f)-\frac{d}{n}I\right)}
\end{array}
\end{equation}
is a {\bf moment} map of ${\bf P}_n^d$ under the ${\rm SU}(n,{\mathbb C})$-action.
\end{theorem}

\begin{remark} The norm and the inner product on the right hand side of (\ref{Hesse}) are defined as in Lemma \ref{innerproduct}. We will not introduce the general theory of symplectic reduction and its relation to geometric invariant theory, so Theorem \ref{moment} should be regard as the definition of the moment map here. For more details on symplectic geometry, see \cite{LOSG}. But we should notice that the original definition of the moment map is a map ${\frak m}:{\bf P}_n^d\rightarrow {\frak s}{\frak u}(n)^\vee$ where ${\frak s}{\frak u}(n)^\vee = {\rm Hom}_{\mathbb C}({\frak s}{\frak u}(n), {\mathbb C})$ is the dual vector space of ${\frak s}{\frak u}(n)$. Since ${\frak s}{\frak l}(n,{\mathbb C})$ is semi-simple, the Killing form $B(\alpha, \beta)={\rm Re}({\rm tr}(\alpha\beta)) \ \forall \alpha, \beta \in {\frak s}{\frak l}(n,{\mathbb C}) $ is non-degenerate and gives an identification ${\frak s}{\frak l}(n,{\mathbb C})\rightarrow {\frak s}{\frak l}(n,{\mathbb C})^\vee$, $\alpha \mapsto B(\alpha , \cdot )/B(\alpha ,\alpha)$. Under this isomorphism, a skew-hermitian matrix will be identified to a hermitian matrix, so ${\frak s}{\frak u}(n)^\vee \simeq \sqrt{-1}{\frak s}{\frak u}(n)$ where $\sqrt{-1}{\frak s}{\frak u}(n)$ is the set of hermitian matrices of trace 0.
\end{remark}

\begin{example}[see \cite{Ness}, Section 10]\label{Example} Let $\alpha = (i_1,\ldots ,i_n)\in {\mathscr W}_d$, and ${\bf x}^\alpha=x_1^{i_1}\cdots x_n^{i_n}\in {\mathscr M}_d$, then
$H({\bf x}^\alpha)={\rm diag}(i_1,\ldots i_n)$ where ${\rm diag}(i_1,\ldots i_n)$ is the diagonal matrix $$\begin{pmatrix}
i_1 & & & \\ & i_2 & & \\ & & \ddots & \\ &&& i_n
\end{pmatrix}.
$$
Thus, let \begin{equation}\label{A} {\mathbb A}=\{\, {\frak m}({\bf x}^\alpha) \, | \, x^\alpha \in {\mathscr M}_d\,\}\end{equation}
then ${\mathbb A}=\{\, {\rm diag}(i_1-d/n, \ldots ,i_n-d/n)\ | \ (i_1,\ldots ,i_n)\in {\mathscr W}_d\, \}$.
\end{example}

\begin{proof} Let $\delta_{lk}$ be the Kronecker notation of $l$ and $k$. 
\begin{align*}
H({\bf x}^\alpha)_{lk} =& H(x_1^{i_1}\cdots x_n^{i_n})_{lk} \\
=& \frac{1}{d\|{\bf x}\|^\alpha} \left\langle \frac{\partial x_1^{i_1}\cdots x_n^{i_n}}{x_l}, \frac{\partial x_1^{i_1}\cdots x_n^{i_n}}{x_k}\right\rangle \\
=& \frac{1}{d} \cdot \frac{1}{\frac{i_1!\cdots i_n!}{d!}} \left\langle i_l x_1^{i_1}\cdots x_l^{i_l-1}\ldots x_n^{i_n} , i_k x_1^{i_1}\cdots x_l^{i_k-1}\ldots x_n^{i_n}\right\rangle \\
=& \frac{(d-1)!}{i_1!\cdots i_n!} \delta_{lk}i_li_k \left\langle  x_1^{i_1}\cdots x_l^{i_l-1}\ldots x_n^{i_n} , x_1^{i_1}\cdots x_l^{i_k-1}\ldots x_n^{i_n}\right\rangle \\
=& \frac{(d-1)!}{i_1!\cdots i_n!} i_l\cdot i_l\delta_{lk} \frac{i_1!\cdots (i_k-1)! \cdots i_n!}{(d-1)!} \\
=& \frac{(d-1)!}{i_1!\cdots i_n!} i_l\cdot \delta_{lk} \frac{i_1!\cdots i_k\cdot (i_k-1)! \cdots i_n!}{(d-1)!} \\
=& i_l\delta_{lk} \qquad \blacksquare
\end{align*}
\end{proof}

In general, let $(V,\|\cdot \|)$ be a metric vector space. For any $x,y\in V$, let $d(x,y)=\|x-y\|$ be the distance from $x$ to $y$. Let $U$ be a subset of $V$, then for any $x\in V$, the distance from $x$ to $U$ is defined as ${\sf d}(x,U)= \inf_{y\in U}d(x,y)$. Now, define the {\bf signed distance} $\overrightarrow{\sf d}(x,U)$ from $x$ to $U$ as
\begin{equation}\label{signed}\overrightarrow{\sf d}(x,U)=\begin{cases}
{\sf d}(x,U)\quad \mbox{ if } x\notin U \\
-{\sf d}(x,U^c)\quad  \mbox{ if } x\in U
\end{cases}\end{equation}
where $U^c=V\backslash U$. For our case, the space ${\frak s}{\frak u}(n, {\mathbb C})$ is equipped with a metric by the Killing form, and let ${\sf d}$ be the distance according to this metric, and $\overrightarrow{\sf d}$ be the signed distance of $\sf d$. For $f\in {\bf P}_n^d$, let ${\cal O}(f)$ be the ${\rm SL}(n,{\mathbb C})$-orbit of $f$ in ${\bf P}_n^d$, and let $\overline{{\cal O}(f)}$ be the closure of ${\cal O}(f)$ in ${\bf P}_n^d$.
\begin{theorem}[See \cite{Dolgachev-Hu}] Let $O$ be the $0$ element of ${\frak s}{\frak u}(n)$. Define the real-valued function
$$\begin{array}{cccc}
M: & {\bf P}_n^d & \longrightarrow & {\mathbb R} \\
   & f & \longmapsto & {\overrightarrow{\sf d}}(O,\overline{{\cal O}(f)}\, ).
\end{array}$$
Then $M$ equals to the Hilbert-Mumford numerical criterion of geometric invariant theory stability function of the ${\rm SL}(n,{\mathbb C})$-action on ${\bf P}_n^d$. 
\end{theorem}

\begin{remark} According to geometric invariant theory \cite{MFK}, the projective variety ${\bf P}_n^d$ will be decomposed as a disjoint union ${\bf P}_n^d=\left({\bf P}_n^d\right)^{\rm ss} \sqcup \left({\bf P}_n^d\right)^{\rm us} $ with respect to the right action\footnote{In fact, once this action is fixed, we have chosen the linearization, and the geometric invariant theory stability is determined by the choice of linearization \cite{MFK}.} by ${\rm SL}(n,{\mathbb C})$, where $\left({\bf P}_n^d\right)^{\rm ss}$ is the set of so-called semi-stable points, and $\left({\bf P}_n^d\right)^{\rm us}$ is the set of so-called unstable points. The function $M$ is the Hilbert-Mumford numerical criterion of geometric invariant theory stability function means that 
\begin{align*}
&\left({\bf P}_n^d\right)^{\rm ss}=\{\, f\in {\bf P}_n^d \, | \, M(f)\leq 0\,\}, \\
&\left({\bf P}_n^d\right)^{\rm us}=\{\, f\in {\bf P}_n^d \, | \, M(f)> 0\,\}.
\end{align*}
\end{remark}

Let $T$ be the maximal subtorus of ${\rm SU}(n)$ defined by the diagonal matrices, and let $\frak t$ be the Lie algebra of $T$. Then the inclusion map of Lie algebras ${\frak t}\subset {\frak s}{\frak u}(n)$ induces the restriction map on the dual spaces ${\rm Res}: {\frak s}{\frak u}(n)^\vee \rightarrow {\frak t}^\vee$. We have shown that ${\frak s}{\frak u}(n)^\vee \simeq \sqrt{-1}\, {\frak s}{\frak u}(n)$, thus ${\frak t}^\vee \simeq \sqrt{-1}\, {\frak t}$, and these identifications give a restriction map ${\rm Res}: \sqrt{-1}\, {\frak s}{\frak u}(n) \rightarrow \sqrt{-1}\, {\frak t}$. Let $${\frak m}_T:= {\rm Res}\circ {\frak m}: {\bf P}_n^d\ \longrightarrow\ \sqrt{-1}\, {\frak t}.$$
Since ${\sqrt{-1}}\, {\frak t}$ consists diagonal hermitian matrices, all the matrices in ${\sqrt{-1}}\, {\frak t}$ will have real entries on their diagonals. If we regard ${\rm diag}(\lambda_1,\ldots ,\lambda_n)\in \sqrt{-1}\, {\frak t}$ as the vector $(\lambda_1,\ldots ,\lambda_n)\in {\mathbb R}^n$, then $\sqrt{-1}\, {\frak t}$ is isomorphic to the hyperplane in ${\mathbb R}^n$ defined by $\lambda_1+\cdots +\lambda_n=0$. The restriction of the metric on ${\frak s}{\frak u}(n)$ induced by the Killing form on ${\sqrt{-1}}\, {\frak t}$ is just the Euclidean norm $\|(\lambda_1,\ldots ,\lambda_n)\|^2=\lambda_1^2+\cdots \lambda_n^2$. We say $\alpha \perp \beta $ in ${\sqrt{-1}}\, {\frak t}$ if $\langle \alpha ,\beta \rangle=0$ under this norm.

\begin{theorem}[See \cite{Atiyah}] Let ${\mathbb A}=\{{\frak m}({\bf x}^\alpha)\, | \, {\bf x}^\alpha \in {\mathscr M}_d\}$ as in Example \ref{Example}. Let ${\cal C}({\mathbb A})$ be the convex polytope generated by ${\mathbb A}$ in $\sqrt{-1}\, {\frak t}$, then the image ${\frak m}_T({\bf P}_n^d)={\cal C}({\mathbb A})$.
\end{theorem}

It is very important of study on the critical points of the function $\|{\frak m}_T\|^2:f \mapsto \|{\frak m}_T(f)\|^2$ (See \cite{Kirwan} and Chapter 8 of \cite{MFK}). Choose a Weyl chamber $$\sqrt{-1}\, {\frak t}^+=\{\, {\rm diag}(\lambda_1,\ldots ,\lambda_n)\in {\sqrt{-1}\, {\frak t}}\, | \, \lambda_1\geq \lambda_2 \geq \ldots \geq \lambda_n\, \}$$ in $\sqrt{-1}\, {\frak t}$. Let ${\mathbb A}^{\cal B}$ be the set of {\bf minimal combinations} of ${\mathbb A}$ as Definition \ref{mc} and let ${\mathbb A}^{\cal B}_+={\mathbb A}^{\cal B}\cap \sqrt{-1}\, {\frak t}^+$. For any $\beta \in {\mathbb A}^{\cal B}_+$, define
\begin{equation}\label{Zbeta}
Z_\beta = \left\{\ \sum_{\alpha\in {\mathscr W}_d}c_\alpha {\bf x}^\alpha \  {\big|} \  \mbox{if}\quad H({\bf x}^\alpha)\perp \beta \quad \mbox{then}\quad c_\alpha=0\ \right\}
\end{equation}
and let \begin{align*}C_\beta = &{\rm SU}(n,{\mathbb C}).(Z_\beta \cap {\frak m}^{-1}(\beta))\\  =
& \{\, A.f\ | \ \mbox{for } A \in {\rm SU}(n,{\mathbb C})\mbox{ and } f \in Z_\beta \cap {\frak m}^{-1}(\beta) \}, \end{align*} summarizing the Chapter 3 in \cite{Kirwan}, we have
\begin{lemma}[See \cite{Kirwan}]\label{diagfbeta} The set of critical points of $\|{\frak m}_T\|^2$ equals to the disjoint union $\bigcup_{\beta \in {\mathbb A}^{\cal B}_+} C_\beta$. For any $\beta \in {\mathbb A}^{\cal B}_+$ and for any $f\in C_\beta$, we have $\|{\frak m}_T(f)\|^2=\|\beta \|^2$ and $M(f)=\|\beta\|$. 
\end{lemma}
We need the following formula.
\begin{lemma}[see \cite{Ness}, Corollary 10.1.1]\label{calpha}Let $$f=\sum_{\alpha \in {\mathscr W}_d}c_\alpha \frac{{\bf x}^\alpha}{\|{\bf x}^\alpha\|}.$$ If ${\frak m}(f)$ is diagonal, then $${\frak m}_T(f)={\frak m}(f)=\sum_{\alpha \in {\mathscr W}_d}\frac{|c_\alpha|^2}{\|f\|^2}\, {\frak m}({\bf x}^\alpha).$$
\end{lemma}
Note that $\sum_{\alpha \in {\mathscr W}_d}|c_\alpha|^2=\|f\|^2$, so we have ${\frak m}(f)\in {\cal C}({\mathbb A})$. We prove the following corollary to find $f\in Z_\beta \cap {\frak m}^{-1}(\beta)$ for $\beta \in {\mathbb A}^{\cal B}_+$.
\begin{corollary}\label{Fbeta} Let ${\mathscr U}$ be a non-empty subset of ${\mathscr W}_d$ which satisfies the following conditions:
\begin{enumerate}
\item $S=\{\, {\frak m}({\bf x}^\alpha) \ | \ \alpha \in {\mathscr U}\, \}$ is an affinely independent subset of ${\mathbb A}$,
\item $\beta=\tau(S)\in {\cal C}(S)^\circ$ as the first condition in Theorem \ref{ak}.
\end{enumerate}
Let $\beta=\sum_{\alpha \in {\mathscr U}}q_\alpha \, {\frak m}({\bf x}^\alpha) $, $\sum_{\alpha \in {\mathscr U}}q_\alpha=1$, then
\begin{equation}\label{fbeta}
f_\beta=\sum_{\alpha \in {\mathscr U}}\frac{\sqrt{q_\alpha}}{ \|{\bf x}^\alpha\| }\cdot {\bf x}^\alpha
\end{equation}
is contained in $Z_\beta \cap {\frak m}^{-1}(\beta)$ if ${\frak m}(f_\beta)=\beta$.
\end{corollary}

\begin{proof} Since $\beta =\tau(S)$, we have $\beta \in {\mathbb A}^{\cal B}$, thus $\beta \in {\cal C}({\mathbb A})$, so $\beta$ is diagonal. Also since $\beta = \tau(S)$, we have $\beta \perp {\rm Aff}(S)$, so $\forall \alpha \in {\mathscr U}$, we have $\beta \perp {\frak m}({\bf x}^\alpha)$. Thus 

$$\left\{\, \sum_{\alpha \in {\mathscr U}}c_\alpha {\bf x}^\alpha \ | \ c_\alpha \in {\mathbb C} \right\}\subset Z_\beta.$$
Let $f=\sum_{\alpha \in {\mathscr U}}c'_\alpha {\bf x}^\alpha$, according to Lemma \ref{calpha}, we solve the equation ${\frak m}(f)=\beta$, that is,
$$\sum_{\alpha \in {\mathscr U}}|c'_\alpha|^2\|{\bf x}^\alpha\|^2\, {\bf x}^\alpha = \sum_{\alpha \in {\mathscr U}}q_\alpha \, {\frak m}({\bf x}^\alpha) $$ $\blacksquare$
\end{proof}

The ``$\beta$" in Corollary \ref{Fbeta} satisfies that $\beta \in {\mathbb A}^{\cal B}_k$, where ${\mathbb A}^{\cal B}_k$ is defined as (\ref{akb}) and $k=\sharp(S)$. Let ${\mathbb A}^{\cal B}_{+,k}={\mathbb A}^{\cal B}_+\cap {\mathbb A}^{\cal B}_k$, our goal for studying on ${\bf P}_n^d$ is the following:
\begin{enumerate}
\item For ${\bf P}_n^d$, we compute ${\mathbb A}^{\cal B}_+$ and all positive values of the function $M$;
\item For all $\beta \in {\mathbb A}^{\cal B}_{+,k}$, $k=1,2,\ldots \sharp({\mathbb A})$, we find $f_\beta \in \left({\bf P}_n^d\right)^{\rm us}$ such that ${\frak m}(f_\beta)\in Z_\beta \cap {\frak m}^{-1}(\beta)$ by (\ref{fbeta}).
\end{enumerate}

\section{Stability of Cubic Curves}\label{SOCC}

The situation $n=d=3$ has been studied by \cite{Ness}, in this section we repeat this example by our algorithm.  The computation is complicated, but still could be finished by hand.

In this case, let $R_3={\mathbb C}[x,y,z]$ be the polynomial ring of three variables $x,y,z$, let ${\bf x}=(x,y,z)$.
Let $R_3^3$ be the homogeneous degree 3 part of $R_3$. It is a 10 dimensional vector space spanned by the set of monomials
$${\mathscr M}_3=\{\, x^3,y^3,z^3,x^2y,xy^2,y^2z,yz^2,x^2z,xz^2,xyz\,\}.$$
A general element in $R_3^3$ or ${\bf P}^3_3$ is of the form
\begin{align*}
\sum_{\alpha \in {\mathscr W}_3}c_\alpha {\bf x}^\alpha = & {{c}_{0,0,3}} {{z}^{3}}+{{c}_{0,1,2}} y {{z}^{2}}+{{c}_{1,0,2}} x {{z}^{2}}+{{c}_{0,2,1}} {{y}^{2}} z+{{c}_{1,1,1}} x y z+{{c}_{2,0,1}} {{x}^{2}} z+ \\ & {{c}_{0,3,0}} {{y}^{3}}+{{c}_{1,2,0}} x {{y}^{2}}+{{c}_{2,1,0}} {{x}^{2}} y+{{c}_{3,0,0}} {{x}^{3}}\end{align*}
where ${\mathscr W}_3$ is the set $$\{ \, (3,0,0),(0,3,0),(0,0,3),(2,1,0),(1,2,0),(0,2,1),(0,1,2),(2,0,1),(1,0,2),(1,1,1)\, \}.$$

The Lie algebra $\sqrt{-1}\, {\frak s}{\frak u}(3)$ is visualizable. The set ${\mathscr W}_3$ is regarded as the set of Hessians of monomials $H({\mathscr M}_3)$ in $\sqrt{-1}\, {\frak t}$ that lies in the hyperplane of $${\mathbb R}^3\simeq \{{\rm diag}(\lambda_1,\lambda_2,\lambda_3)|\lambda_i\in {\mathbb R} \}$$ defined by $\lambda_1+\lambda_2+\lambda_3=3$ (see {\bf Fig.} \ref{M33}).
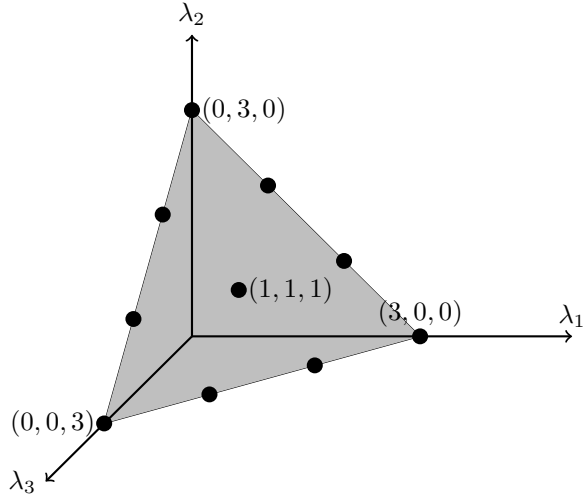
\begin{figure}[h]
\centering
\begin{tikzpicture}

\coordinate (A) at (0,0,0);
\coordinate (B) at (3,0,0);
\coordinate (C) at (0,3,0);
\coordinate (D) at (0,0,3);
\coordinate (E) at (2,1,0);
\coordinate (F) at (1,2,0);
\coordinate (G) at (0,2,1);
\coordinate (H) at (0,1,2);
\coordinate (I) at (2,0,1);
\coordinate (J) at (1,0,2);
\coordinate [label=right:${(1,1,1)}$](K) at (1,1,1);

\draw (B) -- (C) -- (D) -- cycle;
\fill[gray!50](B) -- (C) -- (D) -- cycle;

\draw [thick,->](0,0) -- (xyz cs:x=5);
\draw [thick,->](0,0) -- (xyz cs:y=4);
\draw [thick,->](0,0) -- (xyz cs:z=5);

\coordinate[label=above:$\lambda_1$](L) at (5,0,0);
\coordinate[label=above:$\lambda_2$](L) at (0,4,0);
\coordinate[label=left:$\lambda_3$](L) at (0,0,5);

\foreach \point in {B,C,D,E,F,G,H,I,J,K}
\fill [black] (\point) circle (3pt);

\coordinate [label=right:${(1,1,1)}$](K) at (1,1,1);
\coordinate [label=above:${(3,0,0)}$](B) at (3,0,0);
\coordinate [label=right:${(0,3,0)}$](C) at (0,3,0);
\coordinate [label=left:${(0,0,3)}$](D) at (0,0,3);

\end{tikzpicture}

\caption{The bullets are points in the set $H({\mathscr M}_3)$, and the shaded hyperplane is defined by $\lambda_1+\lambda_2+\lambda_3=3$ \label{M33} }
\end{figure}

The set ${\mathbb A}={\frak m}({\mathscr M}_d)$ of weights is on the hyperplane $\lambda_1+\lambda_2+\lambda_3=0$ (see {\bf Fig.} ). This set also equals to ${\mathbb A}^{\cal B}_1$.
Project the picture of ${\mathbb A}$ to the plane $\lambda_1+\lambda_2+\lambda_3=0$, we have {\bf Fig.} \ref{FIG4}.
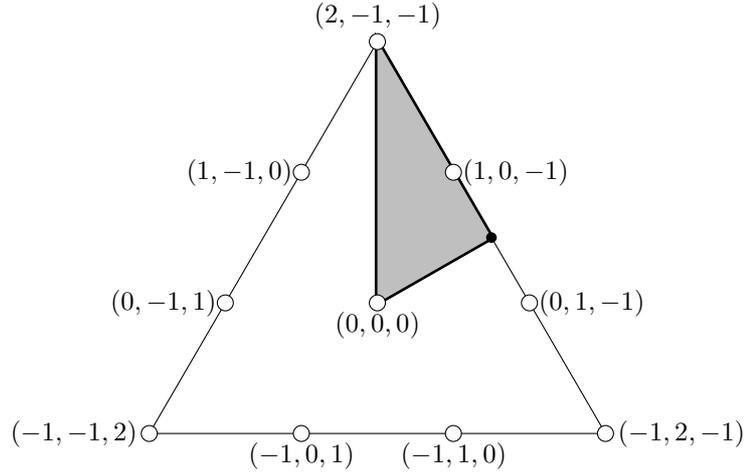
\begin{figure}[H]

\centering
\begin{tikzpicture}
\coordinate [label=left:${}$] (A) at (-1,0);
\coordinate [label=right:${}$] (B) at (5,0);
\draw (A) -- (B);

\node [label=above:${}$] (X) at
($ (A) !  .5!(B) $) {};
\node [fill=red,inner sep=1pt,label=left:${(-1,-1,2)}$] (H) at ($(A)$){};
\node [fill=red,inner sep=1pt,label=right:${(-1,2,-1)}$] (H) at ($(B)$){};

\coordinate [label=above:${}$] (D) at
($ (A) ! .5 ! (B) ! {sin(60)*2} ! 90:(B) $) {};

\node [fill=red,inner sep=1pt,label=above:${(2,-1,-1)}$] (H) at ($(D)$){};

\draw (A) -- (D) -- (B);

\coordinate (F) at
($ (D) ! 2/3!(X)   $) {};

\coordinate [label=below:${(-1,0,1)}$] (E) at ($ (A)! 1/3!(B) $) {};
\coordinate [label=below:${(-1,1,0)}$] (C) at ($ (A)! 2/3!(B) $) {};
\coordinate [label=left:${(0,-1,1)}$] (G) at ($ (A)! 1/3!(D) $) {};
\coordinate [label=left:${(1,-1,0)}$] (H) at ($ (A)! 2/3!(D) $) {};
\coordinate  [label=right:${(1,0,-1)}$] (I) at ($ (D)! 1/3!(B) $) {};
\coordinate  [label=right:${(0,1,-1)}$] (J) at ($ (D)! 2/3!(B) $) {};
\coordinate   (M) at ($ (D)! 1/2!(B) $) {};

\draw[line width=2pt](M) -- (F)  -- (D) -- cycle;

\fill[gray!50](M) -- (F) -- (D) -- cycle;

\coordinate [label=below:${(0,0,0)}$] (F) at (F);
\foreach \point in {A,B,C,D,E,F,G,H,I,J}
\fill [white] (\point) circle (3pt);
\foreach \point in {A,B,C,D,E,F,G,H,I,J}
\draw  (\point) circle (3pt);

\fill [black] (M) circle (2pt);
\end{tikzpicture}

\caption{The set ${\mathbb A}$ and ${\mathbb A}^{\cal B}_{+,1}$\label{FIG4}}
\end{figure}
The shaded part (including the boundary) of {\bf Fig.} \ref{FIG4} is the Weyl chamber defined by $\lambda_1\geq \lambda_2 \geq \lambda_3$. So ${\mathbb A}^{\cal B}_{+,1}$ contains three points $(2,-1,-1)$,$(1,0,-1)$ and $(0,0,0)$. For $d/n=1$, we add the vector $(1,1,1)$ to these three points and get three points $(3,0,0)$, $(2,1,0)$ and $(1,1,1)$ in ${\mathscr W}_d$. They correspond to the monomials $x^3$, $x^2y$ and $xyz$ if we define the lexicographic order $x\succ y \succ z$.

The result for ${\mathbb A}^{\cal B}_{+,1}$ is summarized as Table \ref{k=1}.
The first column lists the elements $\beta\in {\mathbb A}^{\cal B}_{+,1}$. For each $\beta$, the second column lists the affinely independent sets $S$ such that $\tau(S)=\beta$ and $\beta\in {\cal C}(S)^\circ$. It is possible that there exist more than one such $S$ corresponding to the same $\beta$. The polynomial $f\in {\bf P}_3^3$ in the third column are computed according to Corollary \ref{Fbeta}, and $M(f)$ is the length of $\beta$ which also equals to the value of the numerical criterion function $M$.

\begin{table}[h]

\label{tab:1}       % Give a unique label
\begin{tabular}{llll}
\hline\noalign{\smallskip}
$\beta$ & $S$ & $f$ & $M(f)$ \\
\noalign{\smallskip}\hline\noalign{\smallskip}
$(2,-1,-1)$ & $(2,-1,-1)$ & $x^3$ & $\sqrt{6}$ \\
$(1,0,-1)$ & $(1,0,-1)$ & $x^2y$ & $\sqrt{2}$ \\
$(0,0,0)$ & $(0,0,0)$ & $xyz$ & 0\\
\noalign{\smallskip}\hline
\end{tabular}

\caption{Result for ${\mathbb A}^{\cal B}_{+,1}$\label{k=1}}
\end{table}

For $k=2$, there are five elements ${\displaystyle{\left(\frac{1}{2},\frac{1}{2},-1\right)}}$, ${\displaystyle{\left(\frac{1}{2},\frac{1}{2},-1\right)}}$, $ {\displaystyle{\left(\frac{2}{7},\frac{1}{14},-\frac{5}{14}\right)}}$, ${\displaystyle{\left(\frac{1}{2},0,-\frac{1}{2}\right)}}$ and $(0,0,0)$ in ${\mathbb A}^{\cal B}_{+,2}$  as the white circles in Figure \ref{akb+2}. 

\begin{figure}[H]\label{akb+2}

\centering
\begin{tikzpicture}

\coordinate [label=left:${(-1,-1,2)}$] (A) at (-5,0);
\coordinate [label=right:${(-1,2,-1)}$] (B) at (7,0);
\draw (A) -- (B);

\node [label=above:${}$] (X) at
($ (A) !  .5!(B) $) {};

\coordinate [label=above:${(2,-1,-1)}$] (D) at
($ (A) ! .5 ! (B) ! {sin(60)*2} ! 90:(B) $) {};

\coordinate [label=below:${(0,0,0)}$] (F) at
($ (D) ! 2/3!(X)   $) {};

\coordinate [label=right:${\displaystyle{\left(\frac{1}{2},\frac{1}{2},-1\right)}}$] (E) at ($ (B)! 1/2!(D) $) {};

\coordinate   (M) at ($ (D)! 1/2!(B) $) {};
\fill[gray!50](M) -- (F) -- (D) -- cycle;
\draw (A) -- (D) -- (B);
\coordinate [label=below:${}$] (M) at ($ (A)! 1/3!(B) $) {};
\coordinate [label=below:${}$] (N) at ($ (A)! 2/3!(B) $) {};
\coordinate [label=left:${}$] (P) at ($ (A)! 1/3!(D) $) {};
\coordinate [label=left:${}$] (H) at ($ (A)! 2/3!(D) $) {};
\coordinate [label=right:${}$] (X) at ($ (D)! 1/3!(B) $) {};
\coordinate [label=right:${}$] (Y) at ($ (D)! 2/3!(B) $) {};

\draw [dashed] (H)--(B);
\coordinate [label=left:$ {\displaystyle{\left(\frac{2}{7},\frac{1}{14},-\frac{5}{14}\right)}}\quad $] (Z) at ($ (B)! 9/14!(H) $) {};

\draw [dashed](H)--(X);
\coordinate [label=above:${\displaystyle{\left(1,-\frac{1}{2},-\frac{1}{2}\right)}}$] (Q) at ($ (X)! 1/2!(H) $) {};

\draw [dashed](H)--(Y);
\coordinate [label=above:${\displaystyle{\left(\frac{1}{2},0,-\frac{1}{2}\right)}}$] (R) at ($ (Y)! 1/2!(H) $) {};

\draw [dashed](P)--(Y);

\foreach \point in {A,B,D,M,N,P,H,X,Y}
\fill [black] (\point) circle (2pt);

\foreach \point in {E,Z,Q,F,R}
\fill [white] (\point) circle (4pt);
\foreach \point in {E,Z,Q,F,R}
\draw  (\point) circle (4pt);

\end{tikzpicture}

\caption{Elements in ${\mathbb A}^{\cal B}_{+,2}$}
\end{figure}
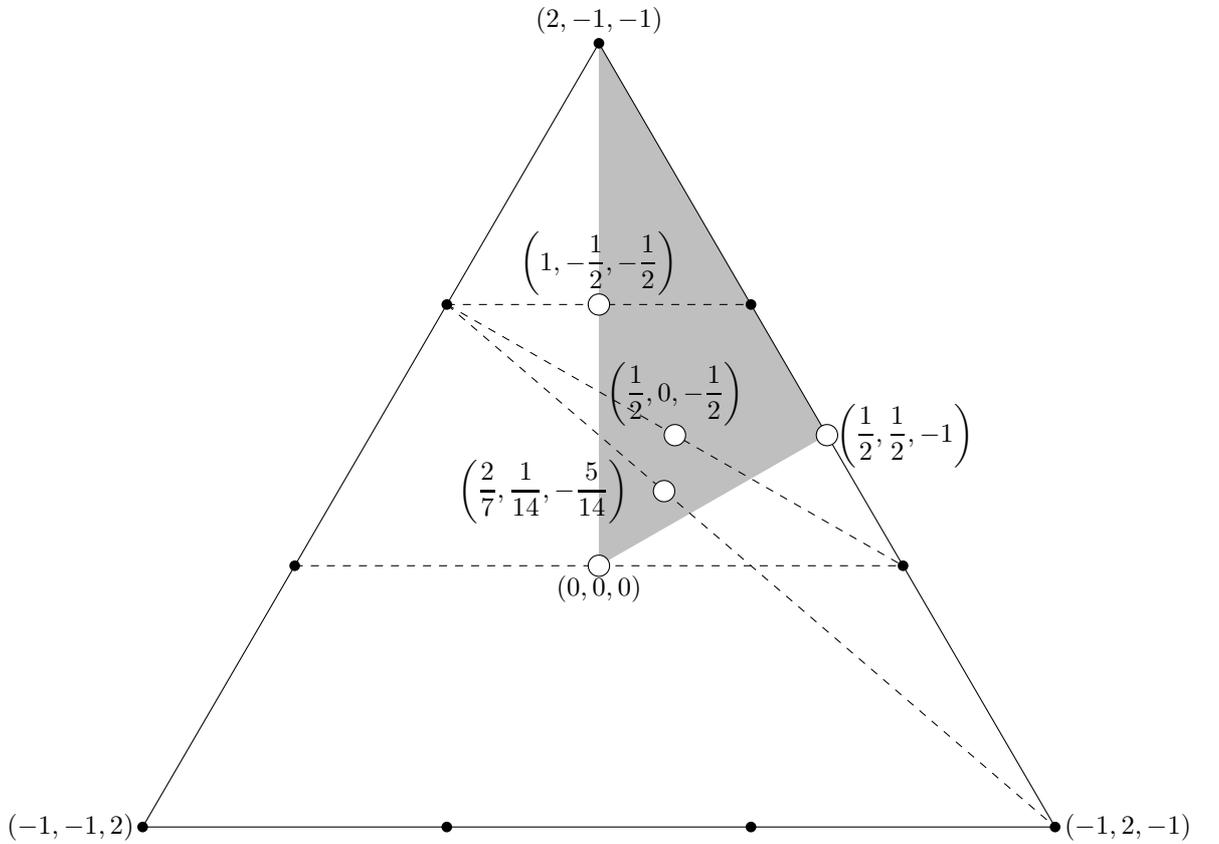

For each $\beta \in {\mathbb A}^{\cal B}_{+,2}$ as circles in the picture, the dashed lines is the hyperplane which contains $\beta$ and perpendicular to $\beta$.

For $\beta= {\displaystyle{\left(\frac{2}{7},\frac{1}{14},-\frac{5}{14}\right)}}$, we have the picture

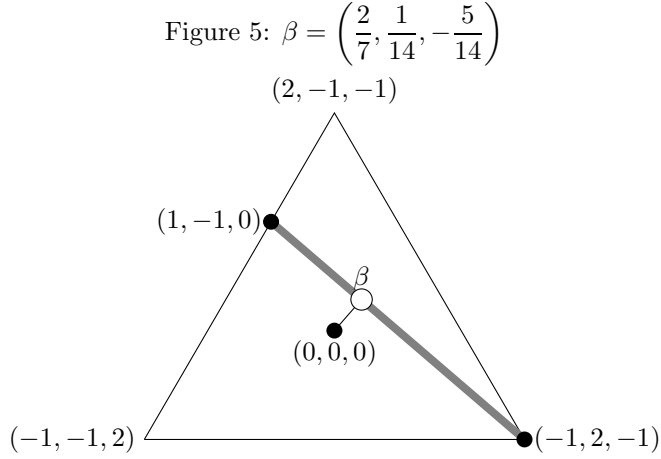
\begin{figure}[H]
\caption{$\beta= {\displaystyle{\left(\frac{2}{7},\frac{1}{14},-\frac{5}{14}\right)}}$}
\centering
\begin{tikzpicture}

\coordinate [label=left:${(-1,-1,2)}$] (A) at (0,0);
\coordinate [label=right:${(-1,2,-1)}$] (B) at (5,0);
\draw (A) -- (B);

\node [label=above:${}$] (X) at
($ (A) !  .5!(B) $) {};

\coordinate [label=above:${(2,-1,-1)}$] (D) at
($ (A) ! .5 ! (B) ! {sin(60)*2} ! 90:(B) $) {};

\draw (A) -- (D) -- (B);

\coordinate [label=below:${(0,0,0)}$] (F) at
($ (D) ! 2/3!(X)   $) {};

\coordinate  (E) at ($ (B)! 1/2!(D) $) {};

\coordinate [label=below:${}$] (M) at ($ (A)! 1/3!(B) $) {};
\coordinate [label=below:${}$] (N) at ($ (A)! 2/3!(B) $) {};
\coordinate [label=left:${}$] (P) at ($ (A)! 1/3!(D) $) {};
\coordinate [label=left:${(1,-1,0)}$] (H) at ($ (A)! 2/3!(D) $) {};
\coordinate [label=right:${}$] (X) at ($ (D)! 1/3!(B) $) {};
\coordinate [label=right:${}$] (Y) at ($ (D)! 2/3!(B) $) {};

\draw [line width=3pt,gray] (H)--(B);
\coordinate [label=above:$\beta$] (Z) at ($ (B)! 9/14!(H) $) {};

\draw (Z)--(F);

\foreach \point in {B,H,F}
\fill [black] (\point) circle (3pt);

\foreach \point in {Z}
\fill [white] (\point) circle (4pt);
\foreach \point in {Z}
\draw  (\point) circle (4pt);

\end{tikzpicture}
\end{figure}

The grey line is the hyperplane which contains $\beta$ and perpendicular to the vector $\beta$ (the direction from $(0,0,0)$ to $\beta$). The line contains 2 points in $(1,-1,0)$ and $(-1,2,-1)$ in ${\mathbb A}$.
They are affinely independent, so $S=\{(1,-1,0),(-1,2,-1)\}$ is the set of cardinality 2 such that $\beta \in {\cal C}(S)^\circ$, the corresponding elements in ${\mathscr W}_d$  are $(2,1,0)=(1,-1,0)+(1,1,1)$ and $(0,3,0)=(-1,2,-1)+(1,1,1)$. Write $\beta$ as a convex combination of these two points as
$${\displaystyle{\beta=\left(\frac{2}{7},\frac{1}{14},-\frac{5}{14}\right)=\frac{5}{14}}\cdot (-1,2,-1) +\frac{9}{14}\cdot (1,-1,0).}$$
The elements of ${\mathscr M}_d$ corresponding to the set $S$ are $x^2z$ and $y^3$. Thus
$$Z_\beta=\{ax^2z+by^3\in {\bf P}_3^3 \ | \ a,b\in {\mathbb C}\}$$ By (\ref{fbeta}), we get that up to $\pm 1$ and $\pm\sqrt{-1}$ multiplication on the coefficients and multiply the maximal common demonstrators,
$$f_\beta= 3\sqrt{3} x^2z + \sqrt{5}y^3.$$
{\bf Warning!} This formula only works under the assumption that ${\frak m}(f_\beta)$ is diagonal, so we have to double check that ${\frak m}(f_\beta)=\beta$. If not, then there is no solution in this convex polytope. But we can check that this is true for $f_\beta=3\sqrt{3} x^2z + \sqrt{5}y^3$

After running the same process for $\beta= {\displaystyle{\left(1,-\frac{1}{2},-\frac{1}{2}\right)}}$, we have
$$f_\beta = x^2y+x^2z $$ in the polytope generated by $(1,0,-1)$ and $(1,-1,0)$. This time, the moment matrix of $f_\beta$ is 
\[\begin{pmatrix}1 & 0 & 0\cr 0 & -\frac{1}{2} & \frac{1}{2}\cr 0 & \frac{1}{2} & -\frac{1}{2}\end{pmatrix}\]
which is NOT diagonal, so we also skip this answer.

\begin{figure}[H]

\centering
\begin{tikzpicture}
{\coordinate [label=left:${}$] (A) at (0,0);
\coordinate [label=right:${(-1,2,-1)}$] (B) at (5,0);
\draw (A) -- (B);

\node [label=above:${}$] (X) at
($ (A) !  .5!(B) $) {};

\coordinate [label=above:${(2,-1,-1)}$] (D) at
($ (A) ! .5 ! (B) ! {sin(60)*2} ! 90:(B) $) {};

\draw (A) -- (D) -- (B);

\coordinate [label=below:${(0,0,0)}$] (F) at
($ (D) ! 2/3!(X)   $) {};

\coordinate  (E) at ($ (B)! 1/2!(D) $) {};

\coordinate [label=below:${}$] (M) at ($ (A)! 1/3!(B) $) {};
\coordinate [label=below:${}$] (N) at ($ (A)! 2/3!(B) $) {};
\coordinate [label=left:${}$] (P) at ($ (A)! 1/3!(D) $) {};
\coordinate [label=left:${}$] (H) at ($ (A)! 2/3!(D) $) {};
\coordinate [label=right:${(1,0,-1)}$] (X) at ($ (D)! 1/3!(B) $) {};
\coordinate [label=right:${(0,1,-1)}$] (Y) at ($ (D)! 2/3!(B) $) {};

\draw [line width=3pt,gray] (D)--(B);
\coordinate [label=left:$\beta\ $] (Z) at ($ (B)! 1/2!(D) $) {};

\draw (Z)--(F);

\foreach \point in {F,B,D,X,Y}
\fill [black] (\point) circle (3pt);

\foreach \point in {Z}
\fill [white] (\point) circle (4pt);
\foreach \point in {Z}
\draw  (\point) circle (4pt);}
\end{tikzpicture}

\caption{$\beta= {\displaystyle{\left(\frac{1}{2},\frac{1}{2},-1\right)}}$}\label{1/2,1/2,-1-1}
\end{figure}
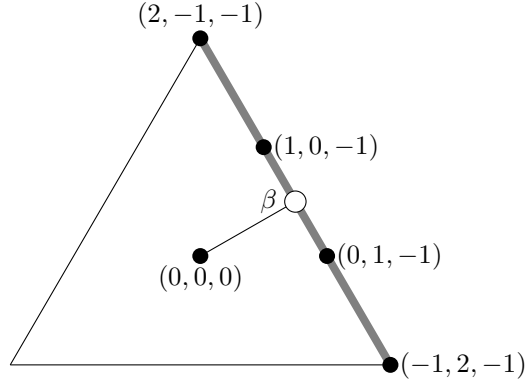
Coming to $\beta= {\displaystyle{\left(\frac{1}{2},\frac{1}{2},-1\right)}}$. This situation is different from the previous two. First, the corresponding elements are $x^3,x^2y,xy^2,y^3$. Thus linear combinations of them do not contain the variable $y$. According to \cite{Ness}, these points come from the moment map of lower dimension. Here they come from ${\rm SL}(2,{\mathbb C})$-representations. They are binary forms of degree 3. Second, there are four points on the hyperplane which is perpendicular to $\beta$ (see Figure\ref{1/2,1/2,-1-1}), so 
$$Z_\beta = \{ax^3+bx^2y+cxy^2+dy^3\in \subset {\bf P}_3^3 \ | \ a,b,c,d\in {\mathbb C}\}$$
is isomorphic to ${\bf P}_2^2\simeq {\mathbb P}^3$ lies in ${\bf P}_3^3$ defined by the equation system $c_{\alpha}=0$ for $\alpha \perp \beta$. 

But consider $\beta$ as an element in ${\mathbb A}^{\cal B}_{+,2}$, so we only consider polytopes which are generated by 2 elements in ${\mathbb A}$ containing $\beta$. There are four such sets $S\subset {\mathbb A} $ of cardinality 2 satisfies this condition, the sets $\{(2,-1,-1),(0,1,-1)\}$, $\{(2,-1,-1),(-1,2,-1)\}$, $\{(1,0,1),(0,1,-1)\}$ and $\{(2,-1,-1),(-1,2,-1)\}$ as Figure \ref{1/2,1/2.-1-2}.

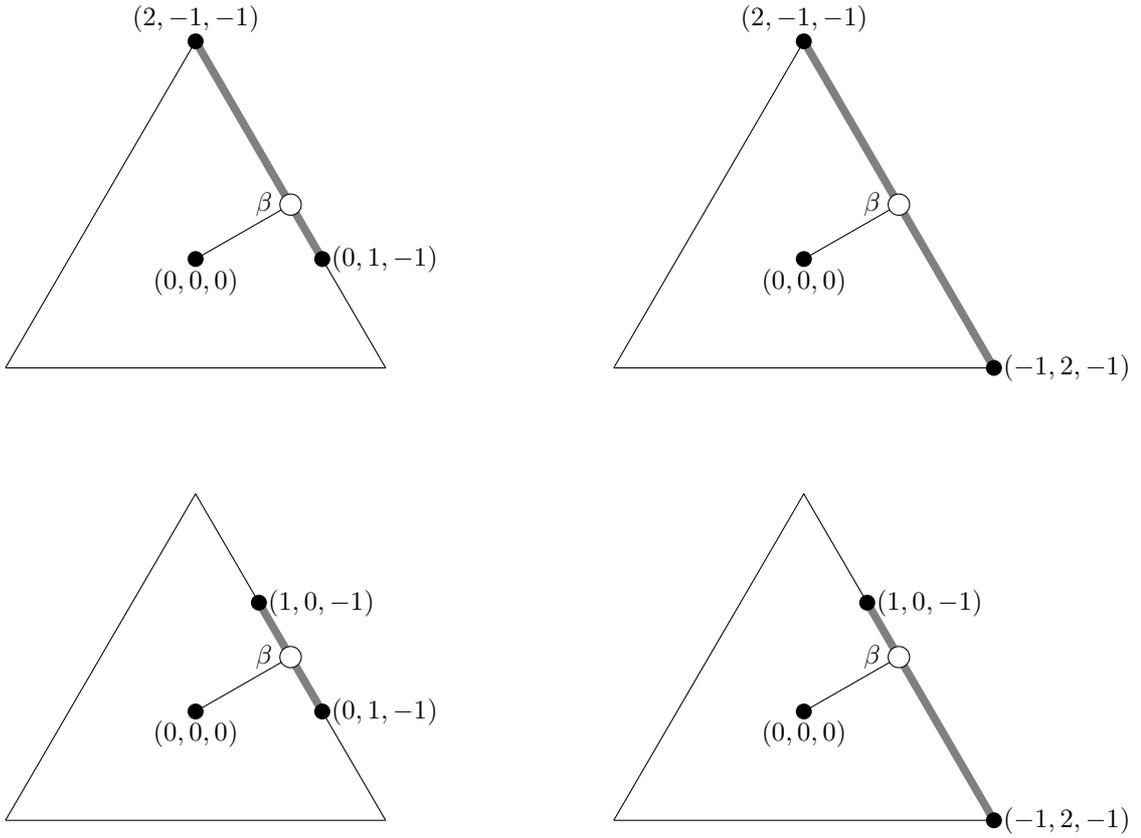
\begin{figure}[h]

\begin{tikzpicture}
{\coordinate [label=left:${}$] (A) at (-10,0);
\coordinate [label=right:${}$] (B) at (-5,0);
\draw (A) -- (B);

\node [label=above:${}$] (X) at
($ (A) !  .5!(B) $) {};

\coordinate [label=above:${(2,-1,-1)}$] (D) at
($ (A) ! .5 ! (B) ! {sin(60)*2} ! 90:(B) $) {};

\draw (A) -- (D) -- (B);

\coordinate [label=below:${(0,0,0)}$] (F) at
($ (D) ! 2/3!(X)   $) {};

\coordinate  (E) at ($ (B)! 1/2!(D) $) {};

\coordinate [label=below:${}$] (M) at ($ (A)! 1/3!(B) $) {};
\coordinate [label=below:${}$] (N) at ($ (A)! 2/3!(B) $) {};
\coordinate [label=left:${}$] (P) at ($ (A)! 1/3!(D) $) {};
\coordinate [label=left:${}$] (H) at ($ (A)! 2/3!(D) $) {};
\coordinate [label=right:${}$] (X) at ($ (D)! 1/3!(B) $) {};
\coordinate [label=right:${(0,1,-1)}$] (Y) at ($ (D)! 2/3!(B) $) {};

\draw [line width=3pt,gray] (D)--(Y);
\coordinate [label=left:$\beta\ $] (Z) at ($ (B)! 1/2!(D) $) {};

\draw (Z)--(F);

\foreach \point in {F,D,Y}
\fill [black] (\point) circle (3pt);

\foreach \point in {Z}
\fill [white] (\point) circle (4pt);
\foreach \point in {Z}
\draw  (\point) circle (4pt);}

{\coordinate [label=left:${}$] (A) at (-2,0);
\coordinate [label=right:${(-1,2,-1)}$] (B) at (3,0);
\draw (A) -- (B);

\node [label=above:${}$] (X) at
($ (A) !  .5!(B) $) {};

\coordinate [label=above:${(2,-1,-1)}$] (D) at
($ (A) ! .5 ! (B) ! {sin(60)*2} ! 90:(B) $) {};

\draw (A) -- (D) -- (B);

\coordinate [label=below:${(0,0,0)}$] (F) at
($ (D) ! 2/3!(X)   $) {};

\coordinate  (E) at ($ (B)! 1/2!(D) $) {};

\coordinate [label=below:${}$] (M) at ($ (A)! 1/3!(B) $) {};
\coordinate [label=below:${}$] (N) at ($ (A)! 2/3!(B) $) {};
\coordinate [label=left:${}$] (P) at ($ (A)! 1/3!(D) $) {};
\coordinate [label=left:${}$] (H) at ($ (A)! 2/3!(D) $) {};
\coordinate [label=right:${}$] (X) at ($ (D)! 1/3!(B) $) {};
\coordinate [label=right:${}$] (Y) at ($ (D)! 2/3!(B) $) {};

\draw [line width=3pt,gray] (D)--(B);
\coordinate [label=left:$\beta\ $] (Z) at ($ (B)! 1/2!(D) $) {};

\draw (Z)--(F);

\foreach \point in {F,B,D}
\fill [black] (\point) circle (3pt);

\foreach \point in {Z}
\fill [white] (\point) circle (4pt);
\foreach \point in {Z}
\draw  (\point) circle (4pt);}

{\coordinate [label=left:${}$] (A) at (-10,-6);
\coordinate [label=right:${}$] (B) at (-5,-6);
\draw (A) -- (B);

\node [label=above:${}$] (X) at
($ (A) !  .5!(B) $) {};

\coordinate [label=above:${}$] (D) at
($ (A) ! .5 ! (B) ! {sin(60)*2} ! 90:(B) $) {};

\draw (A) -- (D) -- (B);

\coordinate [label=below:${(0,0,0)}$] (F) at
($ (D) ! 2/3!(X)   $) {};

\coordinate  (E) at ($ (B)! 1/2!(D) $) {};

\coordinate [label=below:${}$] (M) at ($ (A)! 1/3!(B) $) {};
\coordinate [label=below:${}$] (N) at ($ (A)! 2/3!(B) $) {};
\coordinate [label=left:${}$] (P) at ($ (A)! 1/3!(D) $) {};
\coordinate [label=left:${}$] (H) at ($ (A)! 2/3!(D) $) {};
\coordinate [label=right:${(1,0,-1)}$] (X) at ($ (D)! 1/3!(B) $) {};
\coordinate [label=right:${(0,1,-1)}$] (Y) at ($ (D)! 2/3!(B) $) {};

\draw [line width=3pt,gray] (X)--(Y);
\coordinate [label=left:$\beta\ $] (Z) at ($ (B)! 1/2!(D) $) {};

\draw (Z)--(F);

\foreach \point in {F,X,Y}
\fill [black] (\point) circle (3pt);

\foreach \point in {Z}
\fill [white] (\point) circle (4pt);
\foreach \point in {Z}
\draw  (\point) circle (4pt);}

{\coordinate [label=left:${}$] (A) at (-2,-6);
\coordinate [label=right:${(-1,2,-1)}$] (B) at (3,-6);
\draw (A) -- (B);

\node [label=above:${}$] (X) at
($ (A) !  .5!(B) $) {};

\coordinate [label=above:${}$] (D) at
($ (A) ! .5 ! (B) ! {sin(60)*2} ! 90:(B) $) {};

\draw (A) -- (D) -- (B);

\coordinate [label=below:${(0,0,0)}$] (F) at
($ (D) ! 2/3!(X)   $) {};

\coordinate  (E) at ($ (B)! 1/2!(D) $) {};

\coordinate [label=below:${}$] (M) at ($ (A)! 1/3!(B) $) {};
\coordinate [label=below:${}$] (N) at ($ (A)! 2/3!(B) $) {};
\coordinate [label=left:${}$] (P) at ($ (A)! 1/3!(D) $) {};
\coordinate [label=left:${}$] (H) at ($ (A)! 2/3!(D) $) {};
\coordinate [label=right:${(1,0,-1)}$] (X) at ($ (D)! 1/3!(B) $) {};
\coordinate [label=right:${}$] (Y) at ($ (D)! 2/3!(B) $) {};

\draw [line width=3pt,gray] (X)--(B);
\coordinate [label=left:$\beta\ $] (Z) at ($ (B)! 1/2!(D) $) {};

\draw (Z)--(F);

\foreach \point in {F,B,X}
\fill [black] (\point) circle (3pt);

\foreach \point in {Z}
\fill [white] (\point) circle (4pt);
\foreach \point in {Z}
\draw  (\point) circle (4pt);}
\end{tikzpicture}

\caption{Polytopes containing $\beta= {\displaystyle{\left(\frac{1}{2},\frac{1}{2},-1\right)}}$ in the relative interior\label{1/2,1/2.-1-2}}
\end{figure}

For each of set $S$ in Figure \ref{1/2,1/2.-1-2}, using (\ref{fbeta}), we have the following results.
\begin{itemize}
\item For $S=\{(2,-1,-1),(0,1,-1)\}$, we  have $f_\beta=x^3+3xy^2$
\item For $S=\{(2,-1,-1),(-1,2,-1)\}$, we have $f_\beta=x^3+y^3$
\item For $S=\{(1,0,-1),(0,1,-1)\}$, we have $f_\beta=x^2y+xy^2$
\item For $S=\{(1,0,-1),(-1,2,-1)\}$, we have $f_\beta=3x^2y+y^3$
\end{itemize}
{\bf Warning!} Here we have to double check the moment matrices of these polynomials, and we have
$${\frak m}(x^2y+xy^2)=\begin{pmatrix}\frac{1}{2} & 1 & 0\cr 1 & \frac{1}{2} & 0\cr 0 & 0 & -1\end{pmatrix}$$
is NOT diagonal, so this is not a solution in $Z_\beta\cap {\frak m}^{-1}(\beta)$. For all the others, we can check that they have moment matrix $\beta$.

For $\beta={\displaystyle{\left(\frac{1}{2},0,-\frac{1}{2}\right)}}$ we have $f_{\beta}=x^2z+xy^2$ whose moment matrix is $\beta$.

For the $\beta=(0,0,0)$, since the inner product of $(0,0,0)$ with any other point is $0$, so $Z_\beta={\bf P}_3^3$, we have that any line passes through $\beta$ is perpendicular to $\beta$. So there are 3 sets of cardinality $2$ satisfies that $\tau(S)=\beta$ and $\beta\in {\cal C}(S)^\circ$, as the three gray segments in Figure \ref{0,0,0}.

\begin{figure}[H]

\centering
\begin{tikzpicture}
\coordinate [label=left:${}$] (A) at (0,0);
\coordinate [label=right:${}$] (B) at (5,0);
\draw (A) -- (B);

\node [label=above:${}$] (X) at
($ (A) !  .5!(B) $) {};

\coordinate [label=above:${}$] (D) at
($ (A) ! .5 ! (B) ! {sin(60)*2} ! 90:(B) $) {};

\draw (A) -- (D) -- (B);

\coordinate (F) at
($ (D) ! 2/3!(X)   $) {};

\coordinate [label=below:${(-1,0,1)}$] (E) at ($ (A)! 1/3!(B) $) {};
\coordinate [label=below:${(-1,1,0)}$] (C) at ($ (A)! 2/3!(B) $) {};
\coordinate [label=left:${(0,-1,1)}$] (G) at ($ (A)! 1/3!(D) $) {};
\coordinate [label=left:${(1,-1,0)}$] (H) at ($ (A)! 2/3!(D) $) {};
\coordinate  [label=right:${(1,0,-1)}$] (I) at ($ (D)! 1/3!(B) $) {};
\coordinate  [label=right:${(0,1,-1)}$] (J) at ($ (D)! 2/3!(B) $) {};
\coordinate   (M) at ($ (D)! 1/2!(B) $) {};

\draw [line width=3pt,gray] (H)--(C);
\draw [line width=3pt,gray](I)--(E);
\draw [line width=3pt,gray] (J)--(G);

\foreach \point in {C,E,G,H,I,J}
\fill [black] (\point) circle (3pt);

\foreach \point in {F}
\fill [white] (\point) circle (4pt);
\foreach \point in {F}
\draw  (\point) circle (4pt);

\end{tikzpicture}

\caption{$\beta= {\displaystyle{\left(0,0,0\right)}}$\label{0,0,0}}
\end{figure}

Using (\ref{fbeta}), we have
\begin{itemize}
\item For $S=\{(1,-1,0),(-1,1,0)\}$, we  have $f_\beta=x^2z+y^2z$
\item For $S=\{(1,0,-1),(-1,0,1)\}$, we have $f_\beta=x^2y+z^2y$
\item For $S=\{(0,-1,1),(0,1,-1)\}$, we have $f_\beta=xz^2+xy^2$
\end{itemize}

All three cubics have moment matrix ${\rm diag}(0,0,0)$, they are all solutions, and they are ${\frak S}_3$-symmetric.

Last we consider ${\mathbb A}^{\cal B}_{+,3}$. The affine cone of any three points that are affinely independent is a 2-plane, so it must be the hyperplane $\lambda_1+\lambda_2+\lambda_3=0$ itself. Thus if $\sharp(S)=3$ such that $S$ is affinely independent, $\tau(S)=\beta$ and $\beta\in {\cal C}(S)^\circ$, then $\beta=(0,0,0)$. After running the same process as above using (\ref{fbeta}), and check if the moment matrix of the $f_\beta$ equals to ${\rm diag}(0,0,0)$, we have only one solution
$$f_\beta=x^3+y^3+z^3.$$

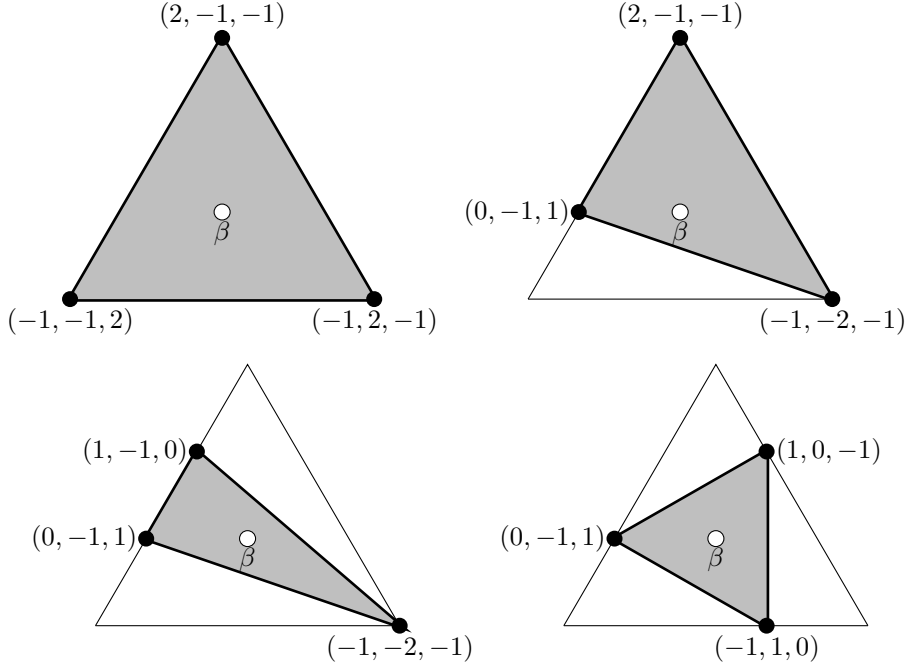
\begin{figure}[H]

\centering
\begin{tikzpicture}
\coordinate [label=below:${(-1,-1,2)}$] (A) at (1,0);
\coordinate [label=below:${(-1,2,-1)}$] (B) at (5,0);
\draw (A) -- (B);

\node [label=above:${}$] (X) at
($ (A) !  .5!(B) $) {};

\coordinate [label=above:${(2,-1,-1)}$] (D) at
($ (A) ! .5 ! (B) ! {sin(60)*2} ! 90:(B) $) {};

\draw (A) -- (D) -- (B);

\coordinate (F) at
($ (D) ! 2/3!(X)   $) {};

\coordinate [label=below:${}$] (E) at ($ (A)! 1/3!(B) $) {};
\coordinate [label=below:${}$] (C) at ($ (A)! 2/3!(B) $) {};
\coordinate [label=left:${}$] (G) at ($ (A)! 1/3!(D) $) {};
\coordinate [label=left:${}$] (H) at ($ (A)! 2/3!(D) $) {};
\coordinate  [label=right:${}$] (I) at ($ (D)! 1/3!(B) $) {};
\coordinate  [label=right:${}$] (J) at ($ (D)! 2/3!(B) $) {};

\draw[line width=2pt](A) -- (B)  -- (D) -- cycle;

\fill[gray!50](A) -- (B) -- (D) -- cycle;

\coordinate [label=below:${\beta}$] (F) at (F);
\foreach \point in {A,B,D}
\fill [black] (\point) circle (3pt);

\fill [white] (F) circle (3pt);
\draw  (F) circle (3pt);

\end{tikzpicture}
\begin{tikzpicture}
\coordinate [label=left:${}$] (A) at (1,0);
\coordinate [label=below:${(-1,-2,-1)}$] (B) at (5,0);
\draw (A) -- (B);

\coordinate [label=above:${(2,-1,-1)}$] (D) at
($ (A) ! .5 ! (B) ! {sin(60)*2} ! 90:(B) $) {};

\draw (A) -- (D) -- (B);

\coordinate (F) at
($ (D) ! 2/3!(X)   $) {};

\coordinate [label=below:${}$] (E) at ($ (A)! 1/3!(B) $) {};
\coordinate [label=below:${}$] (C) at ($ (A)! 2/3!(B) $) {};
\coordinate [label=left:${(0,-1,1)}$] (G) at ($ (A)! 1/3!(D) $) {};
\coordinate [label=left:${}$] (H) at ($ (A)! 2/3!(D) $) {};
\coordinate  [label=right:${}$] (I) at ($ (D)! 1/3!(B) $) {};
\coordinate  [label=right:${}$] (J) at ($ (D)! 2/3!(B) $) {};

\draw[line width=2pt](B) -- (G)  -- (D) -- cycle;

\fill[gray!50](B) -- (G) -- (D) -- cycle;

\coordinate [label=below:${\beta}$] (F) at (F);
\foreach \point in {B,G,D}
\fill [black] (\point) circle (3pt);

\fill [white] (F) circle (3pt);
\draw  (F) circle (3pt);

\end{tikzpicture}
\begin{tikzpicture}
\coordinate [label=left:${}$] (A) at (1,0);
\coordinate [label=below:${(-1,-2,-1)}$] (B) at (5,0);
\draw (A) -- (B);

\node [label=above:${}$] (X) at
($ (A) !  .5!(B) $) {};

\coordinate [label=above:${}$] (D) at
($ (A) ! .5 ! (B) ! {sin(60)*2} ! 90:(B) $) {};

\draw (A) -- (D) -- (B);

\coordinate (F) at
($ (D) ! 2/3!(X)   $) {};

\coordinate [label=below:${}$] (E) at ($ (A)! 1/3!(B) $) {};
\coordinate [label=below:${}$] (C) at ($ (A)! 2/3!(B) $) {};
\coordinate [label=left:${(0,-1,1)}$] (G) at ($ (A)! 1/3!(D) $) {};
\coordinate [label=left:${(1,-1,0)}$] (H) at ($ (A)! 2/3!(D) $) {};
\coordinate  [label=right:${}$] (I) at ($ (D)! 1/3!(B) $) {};
\coordinate  [label=right:${}$] (J) at ($ (D)! 2/3!(B) $) {};
\coordinate   (M) at ($ (D)! 1/2!(B) $) {};

\draw[line width=2pt](B) -- (G)  -- (H) -- cycle;

\fill[gray!50](B) -- (G) -- (H) -- cycle;

\coordinate [label=below:${\beta}$] (F) at (F);

\foreach \point in {B,G,H}
\fill [black] (\point) circle (3pt);

\fill [white] (F) circle (3pt);
\draw  (F) circle (3pt);

\end{tikzpicture}
\begin{tikzpicture}
\coordinate [label=left:${}$] (A) at (1,0);
\coordinate [label=right:${}$] (B) at (5,0);
\draw (A) -- (B);

\node [label=above:${}$] (X) at
($ (A) !  .5!(B) $) {};

\coordinate [label=above:${}$] (D) at
($ (A) ! .5 ! (B) ! {sin(60)*2} ! 90:(B) $) {};

\draw (A) -- (D) -- (B);

\coordinate (F) at
($ (D) ! 2/3!(X)   $) {};

\coordinate [label=below:${}$] (E) at ($ (A)! 1/3!(B) $) {};
\coordinate [label=below:${(-1,1,0)}$] (C) at ($ (A)! 2/3!(B) $) {};
\coordinate [label=left:${(0,-1,1)}$] (G) at ($ (A)! 1/3!(D) $) {};
\coordinate [label=left:${}$] (H) at ($ (A)! 2/3!(D) $) {};
\coordinate  [label=right:${(1,0,-1)}$] (I) at ($ (D)! 1/3!(B) $) {};
\coordinate  [label=right:${}$] (J) at ($ (D)! 2/3!(B) $) {};
\coordinate   (M) at ($ (D)! 1/2!(B) $) {};

\draw[line width=2pt](G) -- (I)  -- (C) -- cycle;

\fill[gray!50](G) -- (I) -- (C) -- cycle;

\coordinate [label=below:${\beta}$] (F) at (F);

\foreach \point in {G,I,C}
\fill [black] (\point) circle (3pt);

\fill [white] (F) circle (3pt);
\draw  (F) circle (3pt);

\end{tikzpicture}

\caption{The set $\beta=(0,0,0)$ as element in  ${\mathbb A}^{\cal B}_{+,3}$}
\end{figure}

\section{Stability of Cubic Surfaces}
\label{SOCS}

This time we consider the situation $n=4$ and $d=3$. Let $R_4={\mathbb C}[x,y,z,w]$ be the polynomial ring of four variables $x,y,z,w$, let ${\bf x}=(x,y,z,w)$. The homogeneous degree 3 part $R_4^3$ is a vector space generated by the set of monomials 
\begin{align*}
{\mathscr M}_4= & \{{{w}^{3}},{{w}^{2}} x,w {{x}^{2}},{{x}^{3}},{{w}^{2}} y,w x y,{{x}^{2}} y,w {{y}^{2}},x {{y}^{2}},{{y}^{3}},{{w}^{2}} z,w x z,{{x}^{2}} z,w y z,x y z,{{y}^{2}} z, \\ & w {{z}^{2}},x {{z}^{2}},y {{z}^{2}},{{z}^{3}}\}. \end{align*}

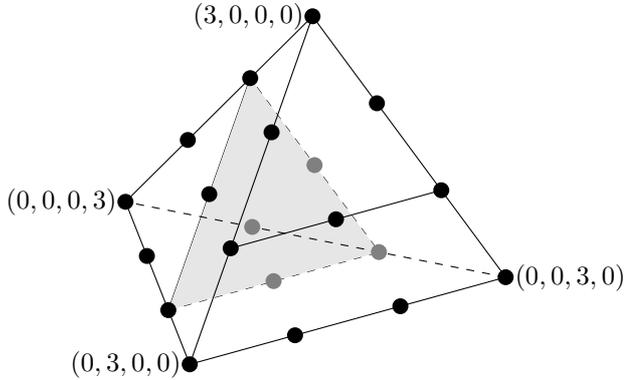
\begin{figure}[H]

\begin{tikzpicture}

\coordinate (A) at (0,0,0);
\coordinate [label=right:${(0,0,3,0)}$] (B) at (3,0,0);
\coordinate [label=left:${(0,0,0,3)}$](C) at (-2,1,0);
\coordinate [label=left:${(0,3,0,0)}$](D) at (0,0,3);
\coordinate (E) at (3,3,0);
\coordinate (F) at (0,3,3);
\coordinate (G) at (3,0,3);
\coordinate [label=left:${(3,0,0,0)}$](H) at (2,5,4);

\coordinate (Y) at (barycentric cs:B=1,C=1,D=1,H=1);
\coordinate (A1) at (barycentric cs:C=1,H=2);
\coordinate (A2) at (barycentric cs:C=2,H=1);
\coordinate (A3) at (barycentric cs:B=1,H=2);
\coordinate (A4) at (barycentric cs:B=2,H=1);
\coordinate (A5) at (barycentric cs:D=2,H=1);
\coordinate (A6) at (barycentric cs:D=1,H=2);
\coordinate (A7) at (barycentric cs:D=1,B=2);
\coordinate (A8) at (barycentric cs:D=2,B=1);
\coordinate (A9) at (barycentric cs:D=1,C=2);
\coordinate (B1) at (barycentric cs:D=2,C=1);
\coordinate (B2) at (barycentric cs:B=2,C=1);
\coordinate (B3) at (barycentric cs:B=1,C=2);
\coordinate (B4) at (barycentric cs:C=1,D=1,H=1);
\coordinate (B5) at (barycentric cs:B=1,D=1,H=1);
\coordinate (B6) at (barycentric cs:B=1,C=1,H=1);
\coordinate (B7) at (barycentric cs:B=1,C=1,D=1);
\draw (A1)--(B1);
\draw [dashed] (B1)--(B2)--(A1);
\fill [gray!20] (B1)--(B2)--(A1)--cycle;
\draw (A4)--(A5);

\draw [dashed] (B)--(C);

\draw (B)--(D);

\draw  (H)--(D);

\draw  (H)--(C);

\draw  (C)--(D);

\draw  (H)--(B);

\foreach \point in {B,C,D,H,A1,A2,A3,A4,A5,A6,A7,A8,A9,B1,B4,B5}
\fill [black] (\point) circle (3pt);

\foreach \point in {B2,B3,B7,B6}
\fill [gray] (\point) circle (3pt);

\end{tikzpicture}
\caption{The set $H({\mathscr M}_4)$}

\end{figure}

We consider the ${\rm SL}(4,{\mathbb C})$(and ${\rm SU}(4)$)-action on the projective space ${\bf P}_4^3$. The elements in the set ${\mathscr W}_4=\{(i_1,i_2,i_3,i_4)\in {\mathbb Z}^4_+ \ | \ i_1+i_2+i_3+i_4 =3 \}$ have four coordinates. How ever they all lie in a hyperplane defined by $i_1+i_2+i_3+i_4 =3$, so we can draw them in three space.

The set of ${\mathbb A}=\{{\frak m}({\bf x}^\alpha)\ | \ \alpha \in {\mathscr W}_4\}$ is as Figure \ref{WETSS}. The shaded part is the Weyl chamber.
\begin{figure}[h]

\begin{tikzpicture}

\coordinate (A) at (0,0,0);
\coordinate [label=right:${(-\frac{3}{4},-\frac{3}{4},\frac{9}{4},-\frac{3}{4})}$] (B) at (6,0,0);
\coordinate [label=left:${(-\frac{3}{4},-\frac{3}{4},-\frac{3}{4},\frac{9}{4})}$](C) at (-4,2,0);
\coordinate [label=left:${(-\frac{3}{4},\frac{9}{4},-\frac{3}{4},-\frac{3}{4})}$](D) at (0,0,6);
\coordinate (E) at (6,6,0);
\coordinate (F) at (0,6,6);
\coordinate (G) at (6,0,6);
\coordinate [label=left:${(\frac{9}{4},-\frac{3}{4},-\frac{3}{4},-\frac{3}{4})}$](H) at (4,10,8);

\coordinate [label=below:${(0,0,0,0)}$](Y) at (barycentric cs:B=1,C=1,D=1,H=1);

\coordinate (A1) at (barycentric cs:C=1,H=2);
\coordinate (A2) at (barycentric cs:C=2,H=1);
\coordinate (A3) at (barycentric cs:B=1,H=2);
\coordinate (A4) at (barycentric cs:B=2,H=1);
\coordinate (A5) at (barycentric cs:D=2,H=1);
\coordinate (A6) at (barycentric cs:D=1,H=2);
\coordinate (A7) at (barycentric cs:D=1,B=2);
\coordinate (A8) at (barycentric cs:D=2,B=1);
\coordinate (A9) at (barycentric cs:D=1,C=2);
\coordinate (B1) at (barycentric cs:D=2,C=1);
\coordinate (B2) at (barycentric cs:B=2,C=1);
\coordinate (B3) at (barycentric cs:B=1,C=2);
\coordinate (B4) at (barycentric cs:C=1,D=1,H=1);
\coordinate (B5) at (barycentric cs:B=1,D=1,H=1);
\coordinate (B6) at (barycentric cs:B=1,C=1,H=1);
\coordinate (B7) at (barycentric cs:B=1,C=1,D=1);
\coordinate (B8) at (barycentric cs:C=1,D=1);
\coordinate (B9) at (barycentric cs:H=1,D=1);
\draw (B8)--(H);
\draw [dashed] (B8)--(B);
\fill [gray!20] (B8)--(B)--(H)--cycle;

\draw [dashed] (B)--(C);

\draw (B)--(D);

\draw  (H)--(D);

\draw  (H)--(C);

\draw  (C)--(D);

\draw  (H)--(B);

\fill [gray!60] (B9)--(H)--(Y)--cycle;
\fill [gray!60] (B5)--(H)--(Y)--cycle;
\draw [dashed] (H)--(B7);
\foreach \point in {B,C,D,H,A1,A2,A3,A4,A5,A6,A7,A8,A9,B1,B4,B5}
\fill [black] (\point) circle (3pt);

\foreach \point in {B2,B3,B7,B6}
\fill [gray] (\point) circle (3pt);

\draw  (H)--(B5)--(B9);
\draw [dashed] (Y)--(B5);
\draw [dashed] (Y)--(B9);
\foreach \point in {B9,Y}
\fill [white] (\point) circle(3pt);
\foreach \point in {B9,Y}
\draw (\point) circle (3pt);

\coordinate [label=below:${(0,0,0,0)}$] (Y) at (Y);

\end{tikzpicture}

\caption{The set ${\mathbb A}$ for cubic surfaces\label{WETSS}}
\end{figure}
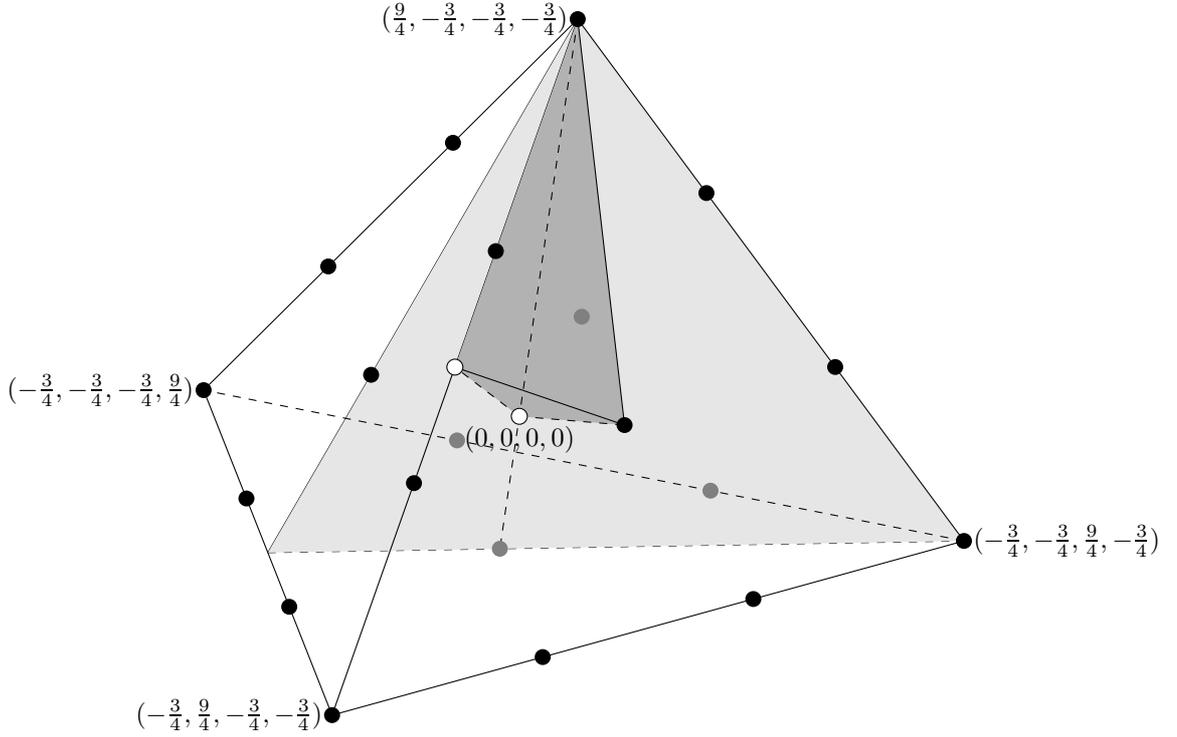

Let us first discuss ${\displaystyle\beta ={\left(\frac{1}{4},\frac{1}{4},\frac{1}{4},-\frac{3}{4}\right)}}$. First $\beta\in {\mathbb A}$ lies on the boundary of the weight polytope ${\cal C}({\mathbb A})$.
 In fact it also lies in the triangle generated by ${\displaystyle{\left(\frac{1}{4},\frac{1}{4},\frac{1}{4},-\frac{3}{4}\right)}}$, ${\displaystyle{\left(-\frac{3}{4},-\frac{3}{4},\frac{9}{4},-\frac{3}{4}\right)}}$ and $\displaystyle{\left(-\frac{3}{4},\frac{9}{4},-\frac{3}{4},-\frac{3}{4}\right)}$.
 They correspond to the monomials $x^3$, $y^3$ and $z^3$.
  Thus if $f\in {\frak m}^{-1}(\beta)$, then $f$ is a convex combination of the monomials that lie in the triangle generated by $x^3$, $y^3$ and $z^3$.
   This means the expression of $f$ only contains three variables $x,y,z$. According to \cite{Ness}, $f$ comes from cubic curves. In precise, the vector space $R_3^3$ is a subspace of $R_4^3$, and this gives an embedding ${\bf P}_3^3\hookrightarrow {\bf P}_4^3$.
    Let $H'$ and ${\frak m}'$ be the Hessian and moment map of ${\bf P}_3^3$ respectively, let $H$ and ${\frak m}$ be the Hessian and moment map of ${\bf P}_4^4$. If $f\in {\bf P}_3^3\hookrightarrow {\bf P}_4^3$, then 
\begin{equation}\label{m3m4}
H = \begin{pmatrix}
H' &  \\
 & 0
\end{pmatrix} \quad \mbox{and} \quad 
{\frak m}(f)=\begin{pmatrix}
{\frak m}'(f) & \\
 & -1
\end{pmatrix} + \frac{1}{4}\, I_{4\times 4}.
\end{equation}
The element ${\displaystyle\beta ={\left(\frac{1}{4},\frac{1}{4},\frac{1}{4},-\frac{3}{4}\right)}}$ can be regarded as a element in ${\mathbb A}^{\cal B}_{1,+}$, ${\mathbb A}^{\cal B}_{2,+}$ and ${\mathbb A}^{\cal B}_{3,+}$ for both ${\frak m}$ and ${\frak m}'$.

\begin{figure}[H]

\centering
\begin{tikzpicture}

\coordinate (A) at (0,0,0);
\coordinate [label=right:${(-\frac{3}{4},-\frac{3}{4},\frac{9}{4},-\frac{3}{4})}$] (B) at (3,0,0);
\coordinate [label=left:${(-\frac{3}{4},-\frac{3}{4},-\frac{3}{4},\frac{9}{4})}$](C) at (-2,1,0);
\coordinate [label=left:${(-\frac{3}{4},\frac{9}{4},-\frac{3}{4},-\frac{3}{4})}$](D) at (0,0,3);
\coordinate (E) at (3,3,0);
\coordinate (F) at (0,3,3);
\coordinate (G) at (3,0,3);
\coordinate [label=left:${(\frac{9}{4},-\frac{3}{4},-\frac{3}{4},-\frac{3}{4})}$](H) at (2,5,4);

\coordinate (A3) at (barycentric cs:B=1,H=2);
\coordinate (A8) at (barycentric cs:D=2,B=1);

\fill [gray!10] (A3)--(A8)--(C)--cycle;

\fill [gray!40] (H)--(B)--(D)--cycle;

\coordinate [label=below:${(0,0,0,0)}$](Y) at (barycentric cs:B=1,C=1,D=1,H=1);
\coordinate [label=left:${}$](A1) at (barycentric cs:C=1,H=2);
\coordinate [label=above:${}$](A2) at (barycentric cs:C=2,H=1);
\coordinate [label=above:${(\frac{5}{4},-\frac{3}{4},\frac{1}{4},-\frac{3}{4})}$](A3) at (barycentric cs:B=1,H=2);
\coordinate [label=above:${}$](A4) at (barycentric cs:B=2,H=1);
\coordinate [label=below:${}$](A5) at (barycentric cs:D=2,H=1);
\coordinate [label=above:${}$](A6) at (barycentric cs:D=1,H=2);
\coordinate [label=below:${}$](A7) at (barycentric cs:D=1,B=2);
\coordinate [label=below:${(-\frac{3}{4},\frac{5}{4},\frac{1}{4},-\frac{3}{4})}$](A8) at (barycentric cs:D=2,B=1);
\coordinate [label=above:${}$](A9) at (barycentric cs:D=1,C=2);
\coordinate [label=left:${}$](B1) at (barycentric cs:D=2,C=1);
\coordinate [label=above:${}$](B2) at (barycentric cs:B=2,C=1);
\coordinate [label=above:${}$](B3) at (barycentric cs:B=1,C=2);
\coordinate [label=above:${}$](B4) at (barycentric cs:C=1,D=1,H=1);
\coordinate [label=right:${}$](B5) at (barycentric cs:B=1,D=1,H=1);
\coordinate [label=above:${}$](B6) at (barycentric cs:B=1,C=1,H=1);
\coordinate [label=above:${}$](B7) at (barycentric cs:B=1,C=1,D=1);
\coordinate [label=above:$\beta $] (Z)  at (barycentric cs:A6=1,A7=1);

\draw [dashed] (B)--(C);

\draw (B)--(D);

\draw  (H)--(D);

\draw  (H)--(C);

\draw [dashed] (A3)--(C);

\draw [dashed] (A8)--(C);

\draw [dashed] (Y)--(C);

\draw  (C)--(D);

\draw  (H)--(B);

\draw [dashed] (Z)--(Y);

\draw [line width=3pt,gray] (A3)--(A8);

\foreach \point in {A3,A8,H,B,D}
\fill [black] (\point) circle (3pt);

\foreach \point in {Y}
\fill [gray] (\point) circle (3pt);

\foreach \point in {Z}
\fill [white] (\point) circle (4pt);
\foreach \point in {Z}
\draw  (\point) circle (4pt);

\end{tikzpicture}
\caption{$\beta= {\displaystyle{\left(\frac{1}{4},\frac{1}{4},\frac{1}{4},-\frac{3}{4}\right)}}$}
\end{figure}
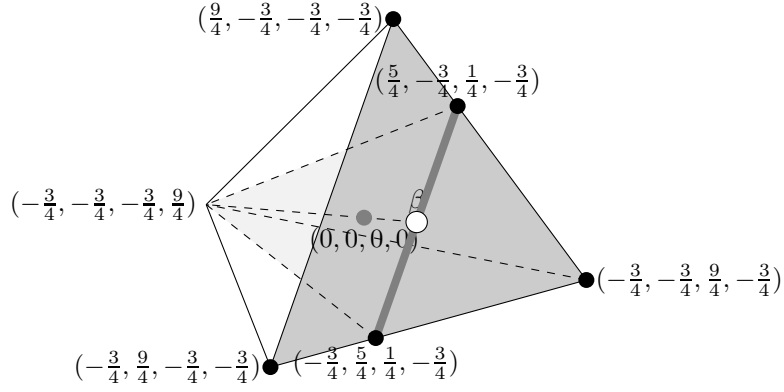

For ${\frak m}'$, we have computed the solutions $f_{\beta'}$ for $\beta'=(0,0,0)$ in the previous section. For $\beta' \in {\mathbb A}^{\cal B}_{1,+}$ we have $f_{\beta'}=xyz$, and ${\frak m}'(f_{\beta'})={\rm diag}(0,0,0)$. By (\ref{m3m4}) we have ${\frak m}(f_{\beta'})=\beta$. The same results are ture for ${\beta'} \in {\mathbb A}^{\cal B}_{2,+}$ and ${\beta'} \in {\mathbb A}^{\cal B}_{3,+}$. As the examples in Figure 11, we have that for $f_{\beta'}=f_\beta = x^2y+z^2y$ and $f_\beta'=f_\beta = x^3+y^3+z^3$, we have ${\frak m}'(f_\beta')=\beta'$ and ${\frak m}(f)=\beta={\displaystyle{\left(\frac{1}{4},\frac{1}{4},\frac{1}{4},-\frac{3}{4}\right)}}$. 

Since $\beta \neq {\rm diag}(0,0,0,0)$, we have if $f_{\beta} \in {\bf P}_4^3$, then $f_\beta \in ({\bf P}_4^3)^{\rm us}$, but for the same $f_{\beta'}\in {\bf P}_3^3$, we have $f_{\beta'}\in ({\bf P}_3^3)^{\rm ss}$ because $M(f_{\beta'})=\|\beta'\|=0$. Thus the stability changes for these critical points of the corresponding moment maps.

Another example we see $\beta$ lies in the boundary of ${\cal C}(\mathbb A)$ is when $\beta = {\displaystyle{\left(\frac{5}{4},-\frac{1}{4},-\frac{1}{4},-\frac{3}{4}\right)}}$.

\begin{figure}[H]

\centering
\begin{tikzpicture}

\coordinate (A) at (0,0,0);
\coordinate [label=right:${(-\frac{3}{4},-\frac{3}{4},\frac{9}{4},-\frac{3}{4})}$] (B) at (3,0,0);
\coordinate [label=left:${(-\frac{3}{4},-\frac{3}{4},-\frac{3}{4},\frac{9}{4})}$](C) at (-2,1,0);
\coordinate [label=left:${(-\frac{3}{4},\frac{9}{4},-\frac{3}{4},-\frac{3}{4})}$](D) at (0,0,3);
\coordinate (E) at (3,3,0);
\coordinate (F) at (0,3,3);
\coordinate (G) at (3,0,3);
\coordinate [label=left:${(\frac{9}{4},-\frac{3}{4},-\frac{3}{4},-\frac{3}{4})}$](H) at (2,5,4);

\coordinate [label=below:${(0,0,0,0)}$](Y) at (barycentric cs:B=1,C=1,D=1,H=1);
\coordinate [label=left:${}$](A1) at (barycentric cs:C=1,H=2);
\coordinate [label=above:${}$](A2) at (barycentric cs:C=2,H=1);
\coordinate [label=right:${(\frac{5}{4},-\frac{3}{4},\frac{1}{4},-\frac{3}{4})}$](A3) at (barycentric cs:B=1,H=2);
\coordinate [label=above:${}$](A4) at (barycentric cs:B=2,H=1);
\coordinate [label=below:${}$](A5) at (barycentric cs:D=2,H=1);
\coordinate [label=left:${(\frac{5}{4},\frac{1}{4},-\frac{3}{4},-\frac{3}{4})}$](A6) at (barycentric cs:D=1,H=2);
\coordinate [label=above:${}$](A7) at (barycentric cs:D=1,B=2);
\coordinate [label=below:${}$](A8) at (barycentric cs:D=2,B=1);
\coordinate [label=above:${}$](A9) at (barycentric cs:D=1,C=2);
\coordinate [label=left:${}$](B1) at (barycentric cs:D=2,C=1);
\coordinate [label=above:${}$](B2) at (barycentric cs:B=2,C=1);
\coordinate [label=above:${}$](B3) at (barycentric cs:B=1,C=2);
\coordinate [label=above:${}$](B4) at (barycentric cs:C=1,D=1,H=1);
\coordinate [label=right:${\left(\frac{1}{4},\frac{1}{4},\frac{1}{4},-\frac{3}{4}\right)}$](B5) at (barycentric cs:B=1,D=1,H=1);
\coordinate [label=above:${}$](B6) at (barycentric cs:B=1,C=1,H=1);
\coordinate [label=above:${}$](B7) at (barycentric cs:B=1,C=1,D=1);
\coordinate [label=above:$\beta $] (Z)  at (barycentric cs:A6=1,A3=1);

\fill [gray!20] (B5)--(C)--(Z)--cycle;

\draw [dashed] (B)--(C);

\draw (B)--(D);

\draw  (H)--(D);

\draw  (H)--(C);

\draw  (B5)--(Z);

\draw [dashed] (B5)--(C)--(Z);

\draw  (C)--(D);

\draw  (H)--(B);

\draw (Z)--(Y);

\draw [line width=3pt,gray] (A6)--(A3);

\foreach \point in {A6,A3,B5}
\fill [black] (\point) circle (3pt);

\foreach \point in {Y}
\fill [gray] (\point) circle (3pt);

\foreach \point in {Z}
\fill [white] (\point) circle (4pt);
\foreach \point in {Z}
\draw  (\point) circle (4pt);

\end{tikzpicture}
\caption{$\beta= {\displaystyle{\left(\frac{5}{4},-\frac{1}{4},-\frac{1}{4},-\frac{3}{4}\right)}}$}
\end{figure}
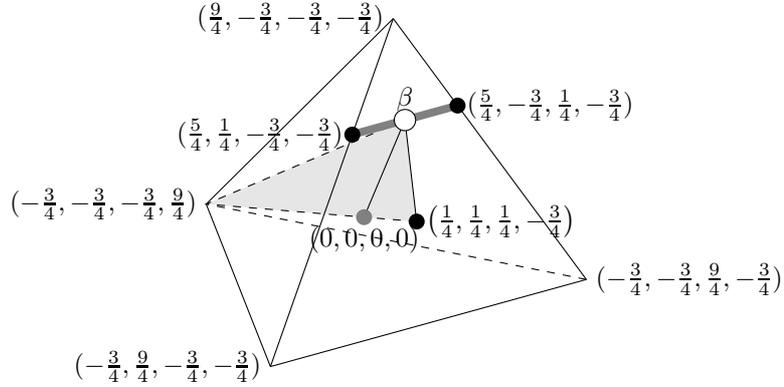

Let $O=(0,0,0,0)$ and $O'={{\displaystyle{\left(\frac{1}{4},\frac{1}{4},\frac{1}{4},-\frac{3}{4}\right)}}}$. Then $O'$ is the origin of ${\frak s \frak u}(3)$, the target of the moment map ${\frak m}'$. The line $OO'$ is perpendicular to the shaded plane, let us call the segment from $\displaystyle{\left(\frac{5}{4},-\frac{3}{4},\frac{1}{4},-\frac{3}{4}\right)}$ to  $\displaystyle{\left(\frac{5}{4},\frac{1}{4},-\frac{3}{4},-\frac{3}{4}\right)}$ be $l$. Thus $O\beta\perp l$ if and only if $O'\beta \perp l$. This implies that $\beta$ is a minimal combination that we have ``already computed" for ${\frak m}'$. 

From now on, we only consider those minimal combinations that lies in the relative interior of ${\cal C}(\mathbb A)$. The critical points which are solved from these minimal combinations will have all variables $x,y,z,w$ in their expressions.

We repeat the same process as before. For 
$\beta= {\displaystyle{\left(\frac{1}{4},-\frac{1}{12},-\frac{1}{12},-\frac{1}{12}\right)}}$, as in Figure 13, we have 
$$f_\beta = xw^2+2xyz.$$
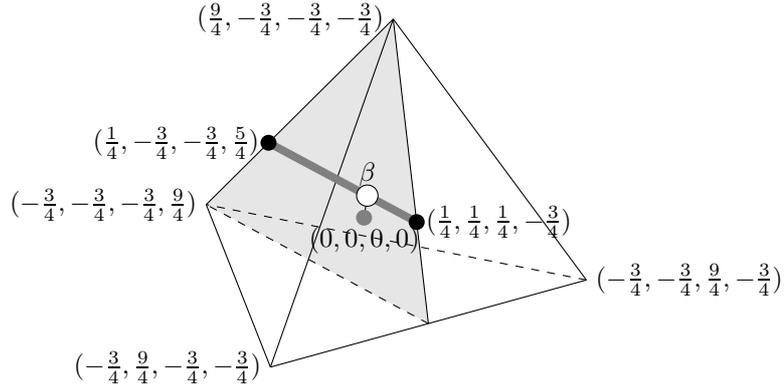
\begin{figure}[H]

\centering

\begin{tikzpicture}

\coordinate (A) at (0,0,0);
\coordinate [label=right:${(-\frac{3}{4},-\frac{3}{4},\frac{9}{4},-\frac{3}{4})}$] (B) at (3,0,0);
\coordinate [label=left:${(-\frac{3}{4},-\frac{3}{4},-\frac{3}{4},\frac{9}{4})}$](C) at (-2,1,0);
\coordinate [label=left:${(-\frac{3}{4},\frac{9}{4},-\frac{3}{4},-\frac{3}{4})}$](D) at (0,0,3);
\coordinate (E) at (3,3,0);
\coordinate (F) at (0,3,3);
\coordinate (G) at (3,0,3);
\coordinate [label=left:${(\frac{9}{4},-\frac{3}{4},-\frac{3}{4},-\frac{3}{4})}$](H) at (2,5,4);

\coordinate [label=below:${(0,0,0,0)}$](Y) at (barycentric cs:B=1,C=1,D=1,H=1);
\coordinate [label=above:${}$](A1) at (barycentric cs:C=1,H=2);
\coordinate [label=left:${(\frac{1}{4},-\frac{3}{4},-\frac{3}{4},\frac{5}{4})}$](A2) at (barycentric cs:C=2,H=1);
\coordinate [label=above:${}$](A3) at (barycentric cs:B=1,H=2);
\coordinate [label=above:${}$](A4) at (barycentric cs:B=2,H=1);
\coordinate [label=below:${}$](A5) at (barycentric cs:D=2,H=1);
\coordinate [label=above:${}$](A6) at (barycentric cs:D=1,H=2);
\coordinate [label=above:${}$](A7) at (barycentric cs:D=1,B=2);
\coordinate [label=above:${}$](A8) at (barycentric cs:D=2,B=1);
\coordinate [label=above:${}$](A9) at (barycentric cs:D=1,C=2);
\coordinate [label=above:${}$](B1) at (barycentric cs:D=2,C=1);
\coordinate [label=above:${}$](B2) at (barycentric cs:B=2,C=1);
\coordinate [label=above:${}$](B3) at (barycentric cs:B=1,C=2);
\coordinate [label=above:${}$](B4) at (barycentric cs:C=1,D=1,H=1);
\coordinate [label=right:${(\frac{1}{4},\frac{1}{4},\frac{1}{4},-\frac{3}{4})}$](B5) at (barycentric cs:B=1,D=1,H=1);
\coordinate [label=above:${}$](B6) at (barycentric cs:B=1,C=1,H=1);
\coordinate [label=above:${}$](B7) at (barycentric cs:B=1,C=1,D=1);
\coordinate [label=above:$\beta $] (Z)  at (barycentric cs:B5=2,A2=1);
\coordinate [label=above:${}$](C1) at (barycentric cs:B=1,D=1);

\fill [gray!20] (C1)--(H)--(C)--cycle;

\draw (C1)--(H);

\draw [dashed] (C1)--(C);

\draw [dashed] (B)--(C);

\draw (B)--(D);

\draw  (H)--(D);

\draw  (H)--(C);

\draw  (C)--(D);

\draw  (H)--(B);

\draw (Z)--(Y);

\draw [line width=3pt,gray] (A2)--(B5);

\foreach \point in {B5,A2}
\fill [black] (\point) circle (3pt);

\foreach \point in {Y}
\fill [gray] (\point) circle (3pt);

\foreach \point in {Z}
\fill [white] (\point) circle (4pt);
\foreach \point in {Z}
\draw  (\point) circle (4pt);

\coordinate [label=above:$\beta $] (Z)  at (Z);
\coordinate [label=below:${(0,0,0,0)}$](Y) at (Y);

\end{tikzpicture}

\caption{$\beta= {\displaystyle{\left(\frac{1}{4},-\frac{1}{12},-\frac{1}{12},-\frac{1}{12}\right)}}$}
\end{figure}

There are 2 other solutions
$$xz^2+2xyw \quad \mbox{and}\quad xy^2+2xzw$$ which are ${\rm GL(4,{\mathbb C})}$(in fact ${\frak S}_4$)-isomorphic to $xw^2+2xyz$ for $\beta\in {\mathbb A}^{\cal B}_{+,2}$.

For $\beta= {\displaystyle{\left(\frac{1}{2},0,0,-\frac{1}{2}\right)}}$ we have 
$$f_\beta = \sqrt{6}xyz + x^2w$$ up to permutation of variables, and 
for $\beta= {\displaystyle{\left(\frac{1}{4},\frac{1}{4},-\frac{1}{4},-\frac{1}{4}\right)}}$ we have 
$$f_\beta = x^2y+z^2w$$ up to permutation of variables. Note that we can solve a solution $f_\beta=xyz+yzw$, but it does not have a diagonal moment matrix, so we skip this answer.

\begin{figure}[H]

\centering

\begin{tikzpicture}

\coordinate (A) at (0,0,0);
\coordinate [label=right:${(-\frac{3}{4},-\frac{3}{4},\frac{9}{4},-\frac{3}{4})}$] (B) at (3,0,0);
\coordinate [label=left:${(-\frac{3}{4},-\frac{3}{4},-\frac{3}{4},\frac{9}{4})}$](C) at (-2,1,0);
\coordinate [label=left:${(-\frac{3}{4},\frac{9}{4},-\frac{3}{4},-\frac{3}{4})}$](D) at (0,0,3);
\coordinate (E) at (3,3,0);
\coordinate (F) at (0,3,3);
\coordinate (G) at (3,0,3);
\coordinate [label=left:${(\frac{9}{4},-\frac{3}{4},-\frac{3}{4},-\frac{3}{4})}$](H) at (2,5,4);

\coordinate [label=below:${(0,0,0,0)}$](Y) at (barycentric cs:B=1,C=1,D=1,H=1);
\coordinate [label=left:${(\frac{5}{4},-\frac{3}{4},-\frac{3}{4},\frac{1}{4})}$](A1) at (barycentric cs:C=1,H=2);
\coordinate [label=above:${}$](A2) at (barycentric cs:C=2,H=1);
\coordinate [label=above:${}$](A3) at (barycentric cs:B=1,H=2);
\coordinate [label=above:${}$](A4) at (barycentric cs:B=2,H=1);
\coordinate [label=below:${}$](A5) at (barycentric cs:D=2,H=1);
\coordinate [label=above:${}$](A6) at (barycentric cs:D=1,H=2);
\coordinate [label=above:${}$](A7) at (barycentric cs:D=1,B=2);
\coordinate [label=above:${}$](A8) at (barycentric cs:D=2,B=1);
\coordinate [label=above:${}$](A9) at (barycentric cs:D=1,C=2);
\coordinate [label=above:${}$](B1) at (barycentric cs:D=2,C=1);
\coordinate [label=above:${}$](B2) at (barycentric cs:B=2,C=1);
\coordinate [label=above:${}$](B3) at (barycentric cs:B=1,C=2);
\coordinate [label=above:${}$](B4) at (barycentric cs:C=1,D=1,H=1);
\coordinate [label=right:${(\frac{1}{4},\frac{1}{4},\frac{1}{4},-\frac{3}{4})}$](B5) at (barycentric cs:B=1,D=1,H=1);
\coordinate [label=above:${}$](B6) at (barycentric cs:B=1,C=1,H=1);
\coordinate [label=above:${}$](B7) at (barycentric cs:B=1,C=1,D=1);
\coordinate [label=above:$\beta $] (Z)  at (barycentric cs:B5=3,A1=1);
\coordinate [label=above:${}$](C1) at (barycentric cs:B=1,D=1);

\fill [gray!20] (C1)--(H)--(C)--cycle;

\draw (C1)--(H);

\draw [dashed] (C1)--(C);

\draw [dashed] (B)--(C);

\draw (B)--(D);

\draw  (H)--(D);

\draw  (H)--(C);

\draw  (C)--(D);

\draw  (H)--(B);

\draw (Z)--(Y);

\draw [line width=3pt,gray] (A1)--(B5);

\foreach \point in {B5,A1}
\fill [black] (\point) circle (3pt);

\foreach \point in {Y}
\fill [gray] (\point) circle (3pt);

\foreach \point in {Z}
\fill [white] (\point) circle (4pt);
\foreach \point in {Z}
\draw  (\point) circle (4pt);
\coordinate [label=left:$\beta\ $] (Z)  at (Z);
\coordinate [label=below:${(0,0,0,0)}$](Y) at (Y);

\end{tikzpicture}
\caption{$\beta= {\displaystyle{\left(\frac{1}{2},0,0,-\frac{1}{2}\right)}}$}
\end{figure}
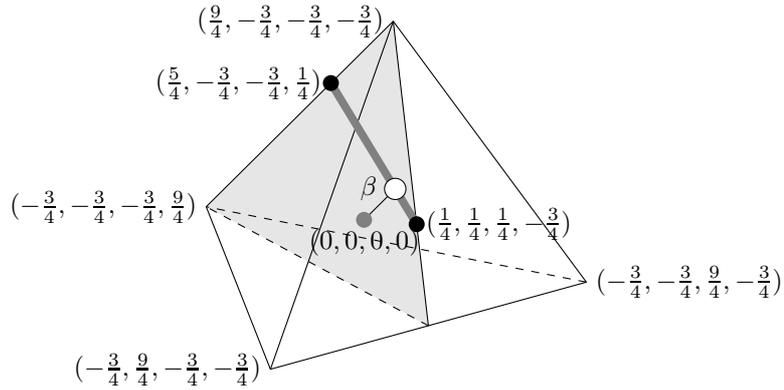

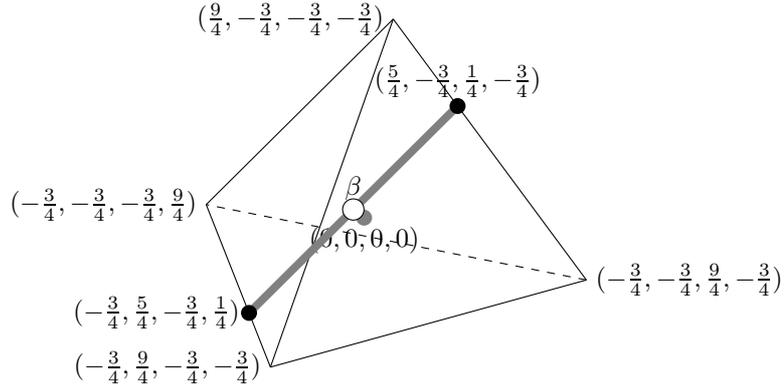
\begin{figure}[H]

\centering
\begin{tikzpicture}

\coordinate (A) at (0,0,0);
\coordinate [label=right:${(-\frac{3}{4},-\frac{3}{4},\frac{9}{4},-\frac{3}{4})}$] (B) at (3,0,0);
\coordinate [label=left:${(-\frac{3}{4},-\frac{3}{4},-\frac{3}{4},\frac{9}{4})}$](C) at (-2,1,0);
\coordinate [label=left:${(-\frac{3}{4},\frac{9}{4},-\frac{3}{4},-\frac{3}{4})}$](D) at (0,0,3);
\coordinate (E) at (3,3,0);
\coordinate (F) at (0,3,3);
\coordinate (G) at (3,0,3);
\coordinate [label=left:${(\frac{9}{4},-\frac{3}{4},-\frac{3}{4},-\frac{3}{4})}$](H) at (2,5,4);

\coordinate [label=below:${(0,0,0,0)}$](Y) at (barycentric cs:B=1,C=1,D=1,H=1);
\coordinate [label=left:${}$](A1) at (barycentric cs:C=1,H=2);
\coordinate [label=above:${}$](A2) at (barycentric cs:C=2,H=1);
\coordinate [label=above:${(\frac{5}{4},-\frac{3}{4},\frac{1}{4},-\frac{3}{4})}$](A3) at (barycentric cs:B=1,H=2);
\coordinate [label=above:${}$](A4) at (barycentric cs:B=2,H=1);
\coordinate [label=below:${}$](A5) at (barycentric cs:D=2,H=1);
\coordinate [label=above:${}$](A6) at (barycentric cs:D=1,H=2);
\coordinate [label=above:${}$](A7) at (barycentric cs:D=1,B=2);
\coordinate [label=above:${}$](A8) at (barycentric cs:D=2,B=1);
\coordinate [label=above:${}$](A9) at (barycentric cs:D=1,C=2);
\coordinate [label=left:${(-\frac{3}{4},\frac{5}{4},-\frac{3}{4},\frac{1}{4})}$](B1) at (barycentric cs:D=2,C=1);
\coordinate [label=above:${}$](B2) at (barycentric cs:B=2,C=1);
\coordinate [label=above:${}$](B3) at (barycentric cs:B=1,C=2);
\coordinate [label=above:${}$](B4) at (barycentric cs:C=1,D=1,H=1);
\coordinate [label=right:${}$](B5) at (barycentric cs:B=1,D=1,H=1);
\coordinate [label=above:${}$](B6) at (barycentric cs:B=1,C=1,H=1);
\coordinate [label=above:${}$](B7) at (barycentric cs:B=1,C=1,D=1);
\coordinate [label=above:$\beta $] (Z)  at (barycentric cs:B5=1,B4=1);

\draw [dashed] (B)--(C);

\draw (B)--(D);

\draw  (H)--(D);

\draw  (H)--(C);

\draw  (C)--(D);

\draw  (H)--(B);

\draw (Z)--(Y);

\draw [line width=3pt,gray] (B1)--(A3);

\foreach \point in {B1,A3}
\fill [black] (\point) circle (3pt);

\foreach \point in {Y}
\fill [gray] (\point) circle (3pt);

\foreach \point in {Z}
\fill [white] (\point) circle (4pt);
\foreach \point in {Z}
\draw  (\point) circle (4pt);

\end{tikzpicture}
\caption{$\beta= {\displaystyle{\left(\frac{1}{4},\frac{1}{4},-\frac{1}{4},-\frac{1}{4}\right)}}$}
\end{figure}

Let us consider ${\mathbb A}^{\cal B}_{+,3}$. As we have discussed for ${\mathbb A}^{\cal B}_{+,2}$, we do not have to consider those $\beta$'s that lie on the boundary of ${\cal C}({\mathbb A})$. Since the process is the same as before, we list the results under the corresponding pictures without reasoning, and skip those solutions which do not have diagonal moment matrices.

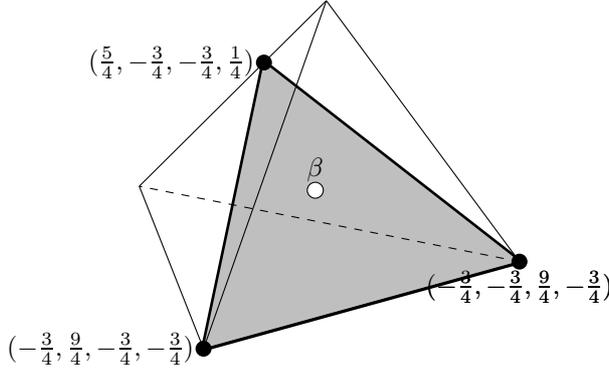
\begin{tikzpicture}

{\coordinate (A) at (0,0,0);
\coordinate [label=below:${(-\frac{3}{4},-\frac{3}{4},\frac{9}{4},-\frac{3}{4})}$] (B) at (3,0,0);
\coordinate (C) at (-2,1,0);
\coordinate [label=left:${(-\frac{3}{4},\frac{9}{4},-\frac{3}{4},-\frac{3}{4})}$](D) at (0,0,3);
\coordinate (E) at (3,3,0);
\coordinate (F) at (0,3,3);
\coordinate (G) at (3,0,3);
\coordinate (H) at (2,5,4);

\coordinate [label=below:${(0,0,0,0)}$](Y) at (barycentric cs:B=1,C=1,D=1,H=1);
\coordinate [label=left:${(\frac{5}{4},-\frac{3}{4},-\frac{3}{4},\frac{1}{4})}$](A1) at (barycentric cs:C=1,H=2);

\draw[line width=2pt](B) -- (D) -- (A1) -- cycle;

\fill[gray!50](B) -- (D) -- (A1) -- cycle;

\coordinate [label=above:${}$](A2) at (barycentric cs:C=2,H=1);
\coordinate [label=above:${}$](A3) at (barycentric cs:B=1,H=2);
\coordinate [label=above:${}$](A4) at (barycentric cs:B=2,H=1);
\coordinate [label=below:${}$](A5) at (barycentric cs:D=2,H=1);
\coordinate [label=above:${}$](A6) at (barycentric cs:D=1,H=2);
\coordinate [label=above:${}$](A7) at (barycentric cs:D=1,B=2);
\coordinate [label=above:${}$](A8) at (barycentric cs:D=2,B=1);
\coordinate [label=above:${}$](A9) at (barycentric cs:D=1,C=2);
\coordinate [label=above:${}$](B1) at (barycentric cs:D=2,C=1);
\coordinate [label=above:${}$](B2) at (barycentric cs:B=2,C=1);
\coordinate [label=above:${}$](B3) at (barycentric cs:B=1,C=2);
\coordinate [label=above:${}$](B4) at (barycentric cs:C=1,D=1,H=1);
\coordinate [label=above:${}$](B5) at (barycentric cs:B=1,D=1,H=1);
\coordinate [label=above:${}$](B6) at (barycentric cs:B=1,C=1,H=1);
\coordinate [label=above:${}$](B7) at (barycentric cs:B=1,C=1,D=1);
\coordinate [label=above:$\beta $] (Z)  at (barycentric cs:B=5,D=5,A1=9);

\draw [dashed] (B)--(C);

\draw (B)--(D);

\draw  (H)--(D);

\draw  (H)--(C);

\draw  (C)--(D);

\draw  (H)--(B);

\foreach \point in {B,D,A1}
\fill [black] (\point) circle (3pt);

\foreach \point in {Z}
\fill [white] (\point) circle (3pt);
\foreach \point in {Z}
\draw  (\point) circle (3pt);

\coordinate [label=right:${f_\beta = \sqrt{5} {{z}^{3}}+\sqrt{5} {{y}^{3}}+3\sqrt{3} w {{x}^{2}}}$](M) at (-2,-2,0);

\coordinate [label=right:${\beta= {\displaystyle{\left(\frac{15}{76},\frac{3}{76},\frac{3}{76},-\frac{21}{76}\right)}}}$,   ](N) at (-2,-3,0) ;

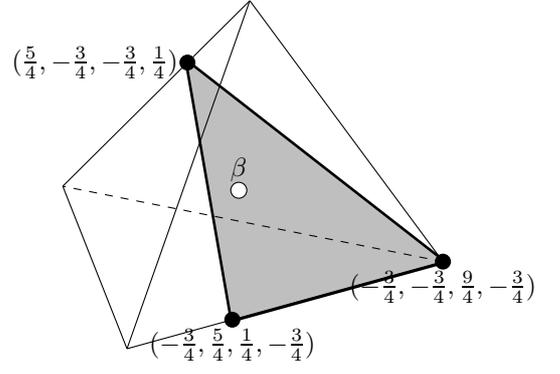
\coordinate [label=below:${(-\frac{3}{4},-\frac{3}{4},\frac{9}{4},-\frac{3}{4})}$] (B) at (3,0,0);}
\end{tikzpicture}
\begin{tikzpicture}

{

\draw[line width=2pt](B) -- (A8) -- (A1) -- cycle;

\fill[gray!50](B) -- (A8) -- (A1) -- cycle;

\coordinate (A) at (0,0,0);
\coordinate [label=below:${(-\frac{3}{4},-\frac{3}{4},\frac{9}{4},-\frac{3}{4})}$] (B) at (3,0,0);
\coordinate (C) at (-2,1,0);
\coordinate (D) at (0,0,3);
\coordinate (E) at (3,3,0);
\coordinate (F) at (0,3,3);
\coordinate (G) at (3,0,3);
\coordinate (H) at (2,5,4);

\coordinate (Y) at (barycentric cs:B=1,C=1,D=1,H=1);
\coordinate [label=left:${(\frac{5}{4},-\frac{3}{4},-\frac{3}{4},\frac{1}{4})}$](A1) at (barycentric cs:C=1,H=2);
\coordinate [label=above:${}$](A2) at (barycentric cs:C=2,H=1);
\coordinate [label=above:${}$](A3) at (barycentric cs:B=1,H=2);
\coordinate [label=above:${}$](A4) at (barycentric cs:B=2,H=1);
\coordinate [label=below:${}$](A5) at (barycentric cs:D=2,H=1);
\coordinate [label=above:${}$](A6) at (barycentric cs:D=1,H=2);
\coordinate [label=above:${}$](A7) at (barycentric cs:D=1,B=2);
\coordinate [label=below:${(-\frac{3}{4},\frac{5}{4},\frac{1}{4},-\frac{3}{4})}$](A8) at (barycentric cs:D=2,B=1);
\coordinate [label=above:${}$](A9) at (barycentric cs:D=1,C=2);
\coordinate [label=above:${}$](B1) at (barycentric cs:D=2,C=1);
\coordinate [label=above:${}$](B2) at (barycentric cs:B=2,C=1);
\coordinate [label=above:${}$](B3) at (barycentric cs:B=1,C=2);
\coordinate [label=above:${}$](B4) at (barycentric cs:C=1,D=1,H=1);
\coordinate [label=above:${}$](B5) at (barycentric cs:B=1,D=1,H=1);
\coordinate [label=above:${}$](B6) at (barycentric cs:B=1,C=1,H=1);
\coordinate [label=above:${}$](B7) at (barycentric cs:B=1,C=1,D=1);
\coordinate [label=above:$\beta $] (Z)  at (barycentric cs:B=5,D=5,A1=9);

\draw [dashed] (B)--(C);

\draw (B)--(D);

\draw  (H)--(D);

\draw  (H)--(C);

\draw  (C)--(D);

\draw  (H)--(B);

\foreach \point in {B,A8,A1}
\fill [black] (\point) circle (3pt);

\foreach \point in {Z}
\fill [white] (\point) circle (3pt);
\foreach \point in {Z}
\draw  (\point) circle (3pt);

\coordinate [label=right:${f_\beta=\sqrt{5} {{z}^{3}}+3 \sqrt{5} {{y}^{2}} z+3 \sqrt{6} w {{x}^{2}}}$](M) at (-2,-2,0);

\coordinate [label=right:${\beta= {\displaystyle{\left(\frac{15}{76},\frac{3}{76},\frac{3}{76},-\frac{21}{76}\right)}}}$,   ](N) at (-2,-3,0) ;}
\end{tikzpicture}
\begin{tikzpicture}
{
\draw[line width=2pt](B5) -- (D) -- (A1) -- cycle;

\fill[gray!50](B5) -- (D) -- (A1) -- cycle;

\coordinate (A) at (0,0,0);
\coordinate  (B) at (3,0,0);
\coordinate (C) at (-2,1,0);
\coordinate [label=left:${(-\frac{3}{4},\frac{9}{4},-\frac{3}{4},-\frac{3}{4})}$](D) at (0,0,3);
\coordinate (E) at (3,3,0);
\coordinate (F) at (0,3,3);
\coordinate (G) at (3,0,3);
\coordinate (H) at (2,5,4);

\coordinate (Y) at (barycentric cs:B=1,C=1,D=1,H=1);
\coordinate [label=left:${(\frac{5}{4},-\frac{3}{4},-\frac{3}{4},\frac{1}{4})}$](A1) at (barycentric cs:C=1,H=2);
\coordinate [label=above:${}$](A2) at (barycentric cs:C=2,H=1);
\coordinate [label=above:${}$](A3) at (barycentric cs:B=1,H=2);
\coordinate [label=above:${}$](A4) at (barycentric cs:B=2,H=1);
\coordinate [label=below:${}$](A5) at (barycentric cs:D=2,H=1);
\coordinate [label=above:${}$](A6) at (barycentric cs:D=1,H=2);
\coordinate [label=below:${}$](A7) at (barycentric cs:D=1,B=2);
\coordinate [label=below:${}$](A8) at (barycentric cs:D=2,B=1);
\coordinate [label=above:${}$](A9) at (barycentric cs:D=1,C=2);
\coordinate [label=above:${}$](B1) at (barycentric cs:D=2,C=1);
\coordinate [label=above:${}$](B2) at (barycentric cs:B=2,C=1);
\coordinate [label=above:${}$](B3) at (barycentric cs:B=1,C=2);
\coordinate [label=above:${}$](B4) at (barycentric cs:C=1,D=1,H=1);
\coordinate [label=right:${(\frac{1}{4},\frac{1}{4},\frac{1}{4},-\frac{3}{4})}$](B5) at (barycentric cs:B=1,D=1,H=1);
\coordinate [label=above:${}$](B6) at (barycentric cs:B=1,C=1,H=1);
\coordinate [label=above:${}$](B7) at (barycentric cs:B=1,C=1,D=1);
\coordinate [label=above:$\beta $] (Z)  at (barycentric cs:B5=6,D=1,A1=3);

\draw [dashed] (B)--(C);

\draw (B)--(D);

\draw  (H)--(D);

\draw  (H)--(C);

\draw  (C)--(D);

\draw  (H)--(B);

\foreach \point in {B5,D,A1}
\fill [black] (\point) circle (3pt);

\foreach \point in {Z}
\fill [white] (\point) circle (3pt);
\foreach \point in {Z}
\draw  (\point) circle (3pt);

\coordinate [label=right:${f_\beta = 6 x y z+{{y}^{3}}+3 w {{x}^{2}}}$](M) at (-2,-2,0);

\coordinate [label=right:${\beta= {\displaystyle{\left(\frac{9}{20},\frac{3}{20},-\frac{3}{20},-\frac{9}{20}\right)}}}$](M) at (-2,-3,0);}
\end{tikzpicture}
\begin{tikzpicture}
{
\draw[line width=2pt](B5) -- (B1) -- (A1) -- cycle;

\fill[gray!50](B5) -- (B1) -- (A1) -- cycle;

\coordinate (A) at (0,0,0);
\coordinate (B) at (3,0,0);
\coordinate (C) at (-2,1,0);
\coordinate (D) at (0,0,3);
\coordinate (E) at (3,3,0);
\coordinate (F) at (0,3,3);
\coordinate (G) at (3,0,3);
\coordinate (H) at (2,5,4);

\coordinate (Y) at (barycentric cs:B=1,C=1,D=1,H=1);
\coordinate [label=left:${(\frac{5}{4},-\frac{3}{4},-\frac{3}{4},\frac{1}{4})}$](A1) at (barycentric cs:C=1,H=2);
%\coordinate [label=above:${(\frac{1}{4},-\frac{3}{4},-\frac{3}{4},\frac{5}{4})}$](A2) at (barycentric cs:C=2,H=1);
%\coordinate [label=above:${(\frac{5}{4},-\frac{3}{4},\frac{1}{4},-\frac{3}{4})}$](A3) at (barycentric cs:B=1,H=2);
%\coordinate [label=above:${(\frac{1}{4},-\frac{3}{4},\frac{5}{4},-\frac{3}{4})}$](A4) at (barycentric cs:B=2,H=1);
%\coordinate [label=above:${(\frac{1}{4},\frac{5}{4},-\frac{3}{4},-\frac{3}{4})}$](A5) at (barycentric cs:D=2,H=1);
%\coordinate [label=above:${(\frac{5}{4},\frac{1}{4},-\frac{3}{4},-\frac{3}{4})}$](A6) at (barycentric cs:D=1,H=2);
%\coordinate [label=above:${(-\frac{3}{4},\frac{1}{4},\frac{5}{4},-\frac{3}{4})}$](A7) at (barycentric cs:D=1,B=2);
%\coordinate [label=above:${(-\frac{3}{4},\frac{5}{4},\frac{1}{4},-\frac{3}{4})}$](A8) at (barycentric cs:D=2,B=1);
%\coordinate [label=above:$A_9$](A9) at (barycentric cs:D=1,C=2);
\coordinate [label=left:${(-\frac{3}{4},\frac{5}{4},-\frac{3}{4},\frac{1}{4})}$](B1) at (barycentric cs:D=2,C=1);
%\coordinate [label=above:$A_{11}$](B2) at (barycentric cs:B=2,C=1);
%\coordinate [label=above:$A_{12}$](B3) at (barycentric cs:B=1,C=2);
%\coordinate [label=above:${(\frac{1}{4},\frac{1}{4},-\frac{3}{4},\frac{1}{4})}$](B4) at (barycentric cs:C=1,D=1,H=1);
\coordinate [label=right:${(\frac{1}{4},\frac{1}{4},\frac{1}{4},-\frac{3}{4})}$](B5) at (barycentric cs:B=1,D=1,H=1);
%\coordinate [label=above:${(\frac{1}{4},-\frac{3}{4},\frac{1}{4},\frac{1}{4})}$](B6) at (barycentric cs:B=1,C=1,H=1);
%\coordinate [label=above:$A_{16}$](B7) at (barycentric cs:B=1,C=1,D=1);

\coordinate [label=above:$\beta $] (Z)  at (barycentric cs:B5=2,B1=1,A1=1);

\draw [dashed] (B)--(C);

\draw (B)--(D);

\draw  (H)--(D);

\draw  (H)--(C);

\draw  (C)--(D);

\draw  (H)--(B);

\foreach \point in {B5,B1,A1}
\fill [black] (\point) circle (3pt);

\foreach \point in {Z}
\fill [white] (\point) circle (3pt);
\foreach \point in {Z}
\draw  (\point) circle (3pt);

\coordinate [label=right:${f_\beta = w{{z}^{2}}+2 x y z+w {{x}^{2}}}$](M) at (-2,-2,0);

\coordinate [label=right:${\beta= {\displaystyle{\left(\frac{1}{4},\frac{1}{4},-\frac{1}{4},-\frac{1}{4}\right)}}}$](M) at (-2,-3,0);}

\end{tikzpicture}

\begin{tikzpicture}

\draw[line width=2pt](D) -- (A4) -- (A1) -- cycle;

\fill[gray!50](D) -- (A4) -- (A1) -- cycle;

\coordinate (A) at (0,0,0);
\coordinate  (B) at (3,0,0);
\coordinate (C) at (-2,1,0);
\coordinate [label=left:${(-\frac{3}{4},\frac{9}{4},-\frac{3}{4},-\frac{3}{4})}$](D) at (0,0,3);
\coordinate (E) at (3,3,0);
\coordinate (F) at (0,3,3);
\coordinate (G) at (3,0,3);
\coordinate (H) at (2,5,4);

\coordinate (Y) at (barycentric cs:B=1,C=1,D=1,H=1);
\coordinate [label=left:${(\frac{5}{4},-\frac{3}{4},-\frac{3}{4},\frac{1}{4})}$](A1) at (barycentric cs:C=1,H=2);
%\coordinate [label=above:${(\frac{1}{4},-\frac{3}{4},-\frac{3}{4},\frac{5}{4})}$](A2) at (barycentric cs:C=2,H=1);
%\coordinate [label=above:${(\frac{5}{4},-\frac{3}{4},\frac{1}{4},-\frac{3}{4})}$](A3) at (barycentric cs:B=1,H=2);
\coordinate [label=above:${(\frac{1}{4},-\frac{3}{4},\frac{5}{4},-\frac{3}{4})}$](A4) at (barycentric cs:B=2,H=1);
%\coordinate [label=above:${(\frac{1}{4},\frac{5}{4},-\frac{3}{4},-\frac{3}{4})}$](A5) at (barycentric cs:D=2,H=1);
%\coordinate [label=above:${(\frac{5}{4},\frac{1}{4},-\frac{3}{4},-\frac{3}{4})}$](A6) at (barycentric cs:D=1,H=2);
%\coordinate [label=above:${(-\frac{3}{4},\frac{1}{4},\frac{5}{4},-\frac{3}{4})}$](A7) at (barycentric cs:D=1,B=2);
%\coordinate [label=above:${(-\frac{3}{4},\frac{5}{4},\frac{1}{4},-\frac{3}{4})}$](A8) at (barycentric cs:D=2,B=1);
%\coordinate [label=above:$A_9$](A9) at (barycentric cs:D=1,C=2);
%\coordinate [label=left:${(-\frac{3}{4},\frac{5}{4},-\frac{3}{4},\frac{1}{4})}$](B1) at (barycentric cs:D=2,C=1);
%\coordinate [label=above:$A_{11}$](B2) at (barycentric cs:B=2,C=1);
%\coordinate [label=above:$A_{12}$](B3) at (barycentric cs:B=1,C=2);
%\coordinate [label=above:${(\frac{1}{4},\frac{1}{4},-\frac{3}{4},\frac{1}{4})}$](B4) at (barycentric cs:C=1,D=1,H=1);
%\coordinate [label=right:${(\frac{1}{4},\frac{1}{4},\frac{1}{4},-\frac{3}{4})}$](B5) at (barycentric cs:B=1,D=1,H=1);
%\coordinate [label=above:${(\frac{1}{4},-\frac{3}{4},\frac{1}{4},\frac{1}{4})}$](B6) at (barycentric cs:B=1,C=1,H=1);
%\coordinate [label=above:$A_{16}$](B7) at (barycentric cs:B=1,C=1,D=1);

\coordinate [label=above:$\beta $] (Z)  at (barycentric cs:D=7,A4=9,A1=9);

\draw [dashed] (B)--(C);

\draw (B)--(D);

\draw  (H)--(D);

\draw  (H)--(C);

\draw  (C)--(D);

\draw  (H)--(B);

\foreach \point in {D,A4,A1}
\fill [black] (\point) circle (3pt);

\foreach \point in {Z}
\fill [white] (\point) circle (3pt);
\foreach \point in {Z}
\draw  (\point) circle (3pt);

\coordinate [label=right:${f_\beta={{3}^{\frac{3}{2}}} x {{z}^{2}}+\sqrt{7} {{y}^{3}}+{{3}^{\frac{3}{2}}} w {{x}^{2}}}$](M) at (-2,-2,0);

\coordinate [label=right:${\beta= {\displaystyle{\left(\frac{33}{100},\frac{9}{100},-\frac{3}{100},-\frac{39}{100}\right)}}}$](M) at (-2,-3,0);

\end{tikzpicture}
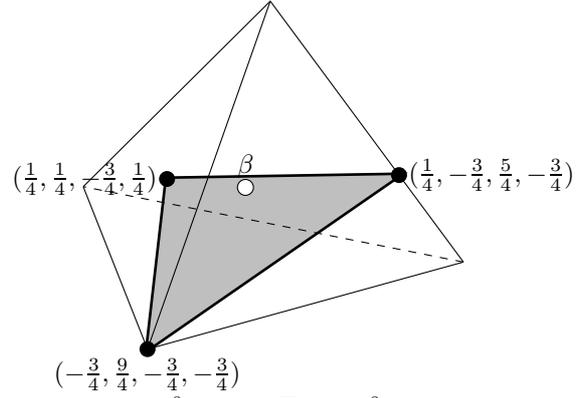
\begin{tikzpicture}

\draw[line width=2pt](D) -- (A4) -- (B4) -- cycle;

\fill[gray!50](D) -- (A4) -- (B4) -- cycle;

\coordinate (A) at (0,0,0);
\coordinate  (B) at (3,0,0);
\coordinate (C) at (-2,1,0);
\coordinate [label=below:${(-\frac{3}{4},\frac{9}{4},-\frac{3}{4},-\frac{3}{4})}$](D) at (0,0,3);
\coordinate (E) at (3,3,0);
\coordinate (F) at (0,3,3);
\coordinate (G) at (3,0,3);
\coordinate (H) at (2,5,4);

\coordinate (Y) at (barycentric cs:B=1,C=1,D=1,H=1);
%\coordinate [label=above:${(\frac{5}{4},-\frac{3}{4},-\frac{3}{4},\frac{1}{4})}$](A1) at (barycentric cs:C=1,H=2);
%\coordinate [label=above:${(\frac{1}{4},-\frac{3}{4},-\frac{3}{4},\frac{5}{4})}$](A2) at (barycentric cs:C=2,H=1);
%\coordinate [label=above:${(\frac{5}{4},-\frac{3}{4},\frac{1}{4},-\frac{3}{4})}$](A3) at (barycentric cs:B=1,H=2);
\coordinate [label=right:${(\frac{1}{4},-\frac{3}{4},\frac{5}{4},-\frac{3}{4})}$](A4) at (barycentric cs:B=2,H=1);
%\coordinate [label=above:${(\frac{1}{4},\frac{5}{4},-\frac{3}{4},-\frac{3}{4})}$](A5) at (barycentric cs:D=2,H=1);
%\coordinate [label=above:${(\frac{5}{4},\frac{1}{4},-\frac{3}{4},-\frac{3}{4})}$](A6) at (barycentric cs:D=1,H=2);
%\coordinate [label=above:${(-\frac{3}{4},\frac{1}{4},\frac{5}{4},-\frac{3}{4})}$](A7) at (barycentric cs:D=1,B=2);
%\coordinate [label=above:${(-\frac{3}{4},\frac{5}{4},\frac{1}{4},-\frac{3}{4})}$](A8) at (barycentric cs:D=2,B=1);
%\coordinate [label=above:$A_9$](A9) at (barycentric cs:D=1,C=2);
%\coordinate [label=left:${(-\frac{3}{4},\frac{5}{4},-\frac{3}{4},\frac{1}{4})}$](B1) at (barycentric cs:D=2,C=1);
%\coordinate [label=above:$A_{11}$](B2) at (barycentric cs:B=2,C=1);
%\coordinate [label=above:$A_{12}$](B3) at (barycentric cs:B=1,C=2);
\coordinate [label=left:${(\frac{1}{4},\frac{1}{4},-\frac{3}{4},\frac{1}{4})}$](B4) at (barycentric cs:C=1,D=1,H=1);
%\coordinate [label=right:${(\frac{1}{4},\frac{1}{4},\frac{1}{4},-\frac{3}{4})}$](B5) at (barycentric cs:B=1,D=1,H=1);
%\coordinate [label=above:${(\frac{1}{4},-\frac{3}{4},\frac{1}{4},\frac{1}{4})}$](B6) at (barycentric cs:B=1,C=1,H=1);
%\coordinate [label=above:$A_{16}$](B7) at (barycentric cs:B=1,C=1,D=1);

\coordinate [label=above:$\beta $] (Z)  at (barycentric cs:D=2,A4=12,B4=21);

\draw [dashed] (B)--(C);

\draw (B)--(D);

\draw  (H)--(D);

\draw  (H)--(C);

\draw  (C)--(D);

\draw  (H)--(B);

\foreach \point in {D,A4,B4}
\fill [black] (\point) circle (3pt);

\foreach \point in {Z}
\fill [white] (\point) circle (3pt);
\foreach \point in {Z}
\draw  (\point) circle (3pt);

\coordinate [label=right:${f_\beta={{3}^{\frac{3}{2}}} x {{z}^{2}}+\sqrt{7} {{y}^{3}}+{{3}^{\frac{3}{2}}} w {{x}^{2}}}$](M) at (-2,-2,0);

\coordinate [label=right:${\beta= {\displaystyle{\left(\frac{27}{140},\frac{3}{140},-\frac{9}{140},-\frac{3}{20}\right)}}}$](M) at (-2,-3,0);
\end{tikzpicture}

\begin{tikzpicture}

\draw[line width=2pt](B1) -- (A4) -- (A1) -- cycle;

\fill[gray!50](B1) -- (A4) -- (A1)-- cycle;

\coordinate (A) at (0,0,0);
\coordinate  (B) at (3,0,0);
\coordinate (C) at (-2,1,0);
\coordinate (D) at (0,0,3);
\coordinate (E) at (3,3,0);
\coordinate (F) at (0,3,3);
\coordinate (G) at (3,0,3);
\coordinate (H) at (2,5,4);

\coordinate (Y) at (barycentric cs:B=1,C=1,D=1,H=1);
\coordinate [label=left:${(\frac{5}{4},-\frac{3}{4},-\frac{3}{4},\frac{1}{4})}$](A1) at (barycentric cs:C=1,H=2);
%\coordinate [label=above:${(\frac{1}{4},-\frac{3}{4},-\frac{3}{4},\frac{5}{4})}$](A2) at (barycentric cs:C=2,H=1);
%\coordinate [label=above:${(\frac{5}{4},-\frac{3}{4},\frac{1}{4},-\frac{3}{4})}$](A3) at (barycentric cs:B=1,H=2);
\coordinate [label=below:${(\frac{1}{4},-\frac{3}{4},\frac{5}{4},-\frac{3}{4})}$](A4) at (barycentric cs:B=2,H=1);
%\coordinate [label=above:${(\frac{1}{4},\frac{5}{4},-\frac{3}{4},-\frac{3}{4})}$](A5) at (barycentric cs:D=2,H=1);
%\coordinate [label=below:${(\frac{5}{4},\frac{1}{4},-\frac{3}{4},-\frac{3}{4})}$](A6) at (barycentric cs:D=1,H=2);
%\coordinate [label=above:${(-\frac{3}{4},\frac{1}{4},\frac{5}{4},-\frac{3}{4})}$](A7) at (barycentric cs:D=1,B=2);
%\coordinate [label=above:${(-\frac{3}{4},\frac{5}{4},\frac{1}{4},-\frac{3}{4})}$](A8) at (barycentric cs:D=2,B=1);
%\coordinate [label=above:$A_9$](A9) at (barycentric cs:D=1,C=2);
\coordinate [label=below:${(-\frac{3}{4},\frac{5}{4},-\frac{3}{4},\frac{1}{4})}$](B1) at (barycentric cs:D=2,C=1);
%\coordinate [label=above:$A_{11}$](B2) at (barycentric cs:B=2,C=1);
%\coordinate [label=above:$A_{12}$](B3) at (barycentric cs:B=1,C=2);
%\coordinate [label=above:${(\frac{1}{4},\frac{1}{4},-\frac{3}{4},\frac{1}{4})}$](B4) at (barycentric cs:C=1,D=1,H=1);
%\coordinate [label=right:${(\frac{1}{4},\frac{1}{4},\frac{1}{4},-\frac{3}{4})}$](B5) at (barycentric cs:B=1,D=1,H=1);
%\coordinate [label=above:${(\frac{1}{4},-\frac{3}{4},\frac{1}{4},\frac{1}{4})}$](B6) at (barycentric cs:B=1,C=1,H=1);
%\coordinate [label=above:$A_{16}$](B7) at (barycentric cs:B=1,C=1,D=1);

\coordinate [label=above:$\beta $] (Z)  at (barycentric cs:B1=9,A4=8,A1=5);

\draw [dashed] (B)--(C);

\draw (B)--(D);

\draw  (H)--(D);

\draw  (H)--(C);

\draw  (C)--(D);

\draw  (H)--(B);

\draw (Z)--(Y);

\foreach \point in {A4,A1,B1}
\fill [black] (\point) circle (3pt);

\foreach \point in {Y}
\fill [gray] (\point) circle (3pt);

\foreach \point in {Z}
\fill [white] (\point) circle (3pt);
\foreach \point in {Z}
\draw  (\point) circle (3pt);

\coordinate [label=right:${f_\beta=2\sqrt{2} x {{z}^{2}}+3 w{{y}^{2}}+\sqrt{5} w{{x}^{2}}}$](M) at (-2,-2,0);

\coordinate [label=right:${\beta= {\displaystyle{\left(\frac{3}{44},\frac{3}{44},-\frac{1}{44},-\frac{5}{44}\right)}}}$](M) at (-2,-3,0);
\end{tikzpicture}
\begin{tikzpicture}

\draw[line width=2pt](A1) -- (A5) -- (A4) -- cycle;

\fill[gray!50](A1) -- (A5) -- (A4)-- cycle;

\coordinate (A) at (0,0,0);
\coordinate (B) at (3,0,0);
\coordinate (C) at (-2,1,0);
\coordinate (D) at (0,0,3);
\coordinate (E) at (3,3,0);
\coordinate (F) at (0,3,3);
\coordinate (G) at (3,0,3);
\coordinate (H) at (2,5,4);

\coordinate (Y) at (barycentric cs:B=1,C=1,D=1,H=1);
\coordinate [label=left:${(\frac{5}{4},-\frac{3}{4},-\frac{3}{4},\frac{1}{4})}$](A1) at (barycentric cs:C=1,H=2);
%\coordinate [label=above:${(\frac{1}{4},-\frac{3}{4},-\frac{3}{4},\frac{5}{4})}$](A2) at (barycentric cs:C=2,H=1);
%\coordinate [label=above:${(\frac{5}{4},-\frac{3}{4},\frac{1}{4},-\frac{3}{4})}$](A3) at (barycentric cs:B=1,H=2);
\coordinate [label=right:${(\frac{1}{4},-\frac{3}{4},\frac{5}{4},-\frac{3}{4})}$](A4) at (barycentric cs:B=2,H=1);
\coordinate [label=left:${(\frac{1}{4},\frac{5}{4},-\frac{3}{4},-\frac{3}{4})}$](A5) at (barycentric cs:D=2,H=1);
%\coordinate [label=below:${(\frac{5}{4},\frac{1}{4},-\frac{3}{4},-\frac{3}{4})}$](A6) at (barycentric cs:D=1,H=2);
%\coordinate [label=below:${(-\frac{3}{4},\frac{1}{4},\frac{5}{4},-\frac{3}{4})}$](A7) at (barycentric cs:D=1,B=2);
%\coordinate [label=above:${(-\frac{3}{4},\frac{5}{4},\frac{1}{4},-\frac{3}{4})}$](A8) at (barycentric cs:D=2,B=1);
%\coordinate [label=above:$A_9$](A9) at (barycentric cs:D=1,C=2);
%\coordinate [label=left:${(-\frac{3}{4},\frac{5}{4},-\frac{3}{4},\frac{1}{4})}$](B1) at (barycentric cs:D=2,C=1);
%\coordinate [label=above:$A_{11}$](B2) at (barycentric cs:B=2,C=1);
%\coordinate [label=above:$A_{12}$](B3) at (barycentric cs:B=1,C=2);
%\coordinate [label=above:${(\frac{1}{4},\frac{1}{4},-\frac{3}{4},\frac{1}{4})}$](B4) at (barycentric cs:C=1,D=1,H=1);
%\coordinate [label=right:${(\frac{1}{4},\frac{1}{4},\frac{1}{4},-\frac{3}{4})}$](B5) at (barycentric cs:B=1,D=1,H=1);
%\coordinate [label=above:${(\frac{1}{4},-\frac{3}{4},\frac{1}{4},\frac{1}{4})}$](B6) at (barycentric cs:B=1,C=1,H=1);
%\coordinate [label=above:$A_{16}$](B7) at (barycentric cs:B=1,C=1,D=1);

\coordinate [label=above:$\beta $] (Z)  at (barycentric cs:A4=3,A5=3,A1=2);

\draw [dashed] (B)--(C);

\draw (B)--(D);

\draw  (H)--(D);

\draw  (H)--(C);

\draw  (C)--(D);

\draw  (H)--(B);

\foreach \point in {A5,A4,A1}
\fill [black] (\point) circle (3pt);

\foreach \point in {Z}
\fill [white] (\point) circle (3pt);
\foreach \point in {Z}
\draw  (\point) circle (3pt);

\coordinate [label=right:${f_\beta=\sqrt{3} x {{z}^{2}}+\sqrt{3} x {{y}^{2}}+\sqrt{2} w {{x}^{2}}}$](M) at (-2,-2,0);

\coordinate [label=right:${\beta= {\displaystyle{\left(\frac{1}{2},0,0,-\frac{1}{2}\right)}}}$](M) at (-2,-3,0);
\end{tikzpicture}

\begin{tikzpicture}

\draw[line width=2pt](A8) -- (A1) -- (A4) -- cycle;

\fill[gray!50](A8) -- (A1) -- (A4)-- cycle;

\coordinate (A) at (0,0,0);
\coordinate  (B) at (3,0,0);
\coordinate (C) at (-2,1,0);
\coordinate (D) at (0,0,3);
\coordinate (E) at (3,3,0);
\coordinate (F) at (0,3,3);
\coordinate (G) at (3,0,3);
\coordinate (H) at (2,5,4);

\coordinate (Y) at (barycentric cs:B=1,C=1,D=1,H=1);
\coordinate [label=left:${(\frac{5}{4},-\frac{3}{4},-\frac{3}{4},\frac{1}{4})}$](A1) at (barycentric cs:C=1,H=2);
%\coordinate [label=above:${(\frac{1}{4},-\frac{3}{4},-\frac{3}{4},\frac{5}{4})}$](A2) at (barycentric cs:C=2,H=1);
%\coordinate [label=above:${(\frac{5}{4},-\frac{3}{4},\frac{1}{4},-\frac{3}{4})}$](A3) at (barycentric cs:B=1,H=2);
\coordinate [label=right:${(\frac{1}{4},-\frac{3}{4},\frac{5}{4},-\frac{3}{4})}$](A4) at (barycentric cs:B=2,H=1);
%\coordinate [label=above:${(\frac{1}{4},\frac{5}{4},-\frac{3}{4},-\frac{3}{4})}$](A5) at (barycentric cs:D=2,H=1);
%\coordinate [label=below:${(\frac{5}{4},\frac{1}{4},-\frac{3}{4},-\frac{3}{4})}$](A6) at (barycentric cs:D=1,H=2);
%\coordinate [label=below:${(-\frac{3}{4},\frac{1}{4},\frac{5}{4},-\frac{3}{4})}$](A7) at (barycentric cs:D=1,B=2);
\coordinate [label=below:$\qquad {(-\frac{3}{4},\frac{5}{4},\frac{1}{4},-\frac{3}{4})}$](A8) at (barycentric cs:D=2,B=1);
%\coordinate [label=above:$A_9$](A9) at (barycentric cs:D=1,C=2);
%\coordinate [label=left:${(-\frac{3}{4},\frac{5}{4},-\frac{3}{4},\frac{1}{4})}$](B1) at (barycentric cs:D=2,C=1);
%\coordinate [label=above:$A_{11}$](B2) at (barycentric cs:B=2,C=1);
%\coordinate [label=above:$A_{12}$](B3) at (barycentric cs:B=1,C=2);
%\coordinate [label=above:${(\frac{1}{4},\frac{1}{4},-\frac{3}{4},\frac{1}{4})}$](B4) at (barycentric cs:C=1,D=1,H=1);
%\coordinate [label=right:${(\frac{1}{4},\frac{1}{4},\frac{1}{4},-\frac{3}{4})}$](B5) at (barycentric cs:B=1,D=1,H=1);
%\coordinate [label=above:${(\frac{1}{4},-\frac{3}{4},\frac{1}{4},\frac{1}{4})}$](B6) at (barycentric cs:B=1,C=1,H=1);
%\coordinate [label=above:$A_{16}$](B7) at (barycentric cs:B=1,C=1,D=1);

\coordinate [label=above:$\beta $] (Z)  at (barycentric cs:A4=1,A1=3,A8=3);

\draw [dashed] (B)--(C);

\draw (B)--(D);

\draw  (H)--(D);

\draw  (H)--(C);

\draw  (C)--(D);

\draw  (H)--(B);

\foreach \point in {A1,A4,A8}
\fill [black] (\point) circle (3pt);

\foreach \point in {Z}
\fill [white] (\point) circle (3pt);
\foreach \point in {Z}
\draw  (\point) circle (3pt);

\coordinate [label=right:${f_\beta=x {{z}^{2}}+\sqrt{3} {{y}^{2}} z+\sqrt{3} w {{x}^{2}}}$](M) at (-2,-2,0);

\coordinate [label=right:${\beta= {\displaystyle{\left(\frac{1}{4},\frac{3}{28},-\frac{1}{28},-\frac{9}{28}\right)}}}$](M) at (-2,-3,0);
\end{tikzpicture}
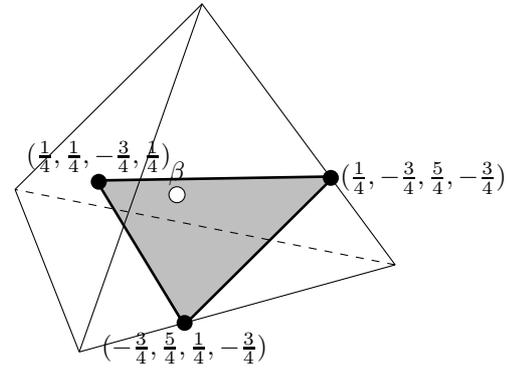
\begin{tikzpicture}

\draw[line width=2pt](A8) -- (B4) -- (A4) -- cycle;

\fill[gray!50](A8) -- (B4) -- (A4)-- cycle;

\coordinate (A) at (0,0,0);
\coordinate  (B) at (3,0,0);
\coordinate (C) at (-2,1,0);
\coordinate (D) at (0,0,3);
\coordinate (E) at (3,3,0);
\coordinate (F) at (0,3,3);
\coordinate (G) at (3,0,3);
\coordinate (H) at (2,5,4);

\coordinate (Y) at (barycentric cs:B=1,C=1,D=1,H=1);
%\coordinate [label=above:${(\frac{5}{4},-\frac{3}{4},-\frac{3}{4},\frac{1}{4})}$](A1) at (barycentric cs:C=1,H=2);
%\coordinate [label=above:${(\frac{1}{4},-\frac{3}{4},-\frac{3}{4},\frac{5}{4})}$](A2) at (barycentric cs:C=2,H=1);
%\coordinate [label=above:${(\frac{5}{4},-\frac{3}{4},\frac{1}{4},-\frac{3}{4})}$](A3) at (barycentric cs:B=1,H=2);
\coordinate [label=right:${(\frac{1}{4},-\frac{3}{4},\frac{5}{4},-\frac{3}{4})}$](A4) at (barycentric cs:B=2,H=1);
%\coordinate [label=above:${(\frac{1}{4},\frac{5}{4},-\frac{3}{4},-\frac{3}{4})}$](A5) at (barycentric cs:D=2,H=1);
%\coordinate [label=below:${(\frac{5}{4},\frac{1}{4},-\frac{3}{4},-\frac{3}{4})}$](A6) at (barycentric cs:D=1,H=2);
%\coordinate [label=below:${(-\frac{3}{4},\frac{1}{4},\frac{5}{4},-\frac{3}{4})}$](A7) at (barycentric cs:D=1,B=2);
\coordinate [label=below:${(-\frac{3}{4},\frac{5}{4},\frac{1}{4},-\frac{3}{4})}$](A8) at (barycentric cs:D=2,B=1);
%\coordinate [label=above:$A_9$](A9) at (barycentric cs:D=1,C=2);
%\coordinate [label=left:${(-\frac{3}{4},\frac{5}{4},-\frac{3}{4},\frac{1}{4})}$](B1) at (barycentric cs:D=2,C=1);
%\coordinate [label=above:$A_{11}$](B2) at (barycentric cs:B=2,C=1);
%\coordinate [label=above:$A_{12}$](B3) at (barycentric cs:B=1,C=2);
\coordinate [label=above:${(\frac{1}{4},\frac{1}{4},-\frac{3}{4},\frac{1}{4})}$](B4) at (barycentric cs:C=1,D=1,H=1);
%\coordinate [label=right:${(\frac{1}{4},\frac{1}{4},\frac{1}{4},-\frac{3}{4})}$](B5) at (barycentric cs:B=1,D=1,H=1);
%\coordinate [label=above:${(\frac{1}{4},-\frac{3}{4},\frac{1}{4},\frac{1}{4})}$](B6) at (barycentric cs:B=1,C=1,H=1);
%\coordinate [label=above:$A_{16}$](B7) at (barycentric cs:B=1,C=1,D=1);

\coordinate [label=above:$\beta $] (Z)  at (barycentric cs:A4=3,B4=6,A8=1);

\draw [dashed] (B)--(C);

\draw (B)--(D);

\draw  (H)--(D);

\draw  (H)--(C);

\draw  (C)--(D);

\draw  (H)--(B);

\foreach \point in {B4,A4,A8}
\fill [black] (\point) circle (3pt);

\foreach \point in {Z}
\fill [white] (\point) circle (3pt);
\foreach \point in {Z}
\draw  (\point) circle (3pt);

\coordinate [label=right:${f_\beta=\sqrt{3} x {{z}^{2}}+{{y}^{2}} z+2 \sqrt{3} w x y}$](M) at (-2,-2,0);

\coordinate [label=right:${\beta= {\displaystyle{\left(\frac{3}{20},\frac{1}{20},-\frac{1}{20},-\frac{3}{20}\right)}}}$](M) at (-2,-3,0);
\end{tikzpicture}

\section{Example of an Affinely Dependent Set}
\label{EAD}
Let us come back to cubic curves, and we discuss a set $S$ which is affinely dependent, hence does not satisfy the condition in Corollary \ref{Fbeta}.

Notation as Section \ref{SOCC}. 

Consider the origin $\beta = (0,0,0)$ as the nearest point form the origin $O=(0,0,0)$ to the convex set generated by the set $S$ of four points
$$S=\{s_1=(-1,-1,2),s_2=(-1,2,-1),s_3=(2,-1,-1),O=(0,0,0)\}.$$ 
\begin{figure}[H]
\centering
\begin{tikzpicture}

\coordinate [label=left:${s_1\ }$] (A) at (0,0);
\coordinate [label=right:${\ s_2}$] (B) at (3,0);
\draw (A) -- (B);

\node [label=above:${}$] (X) at
($ (A) !  .5!(B) $) {};

\coordinate [label=above:${s_3}$] (D) at
($ (A) ! .5 ! (B) ! {sin(60)*2} ! 90:(B) $) {};

\draw (A) -- (D) -- (B);

\coordinate (F) at
($ (D) ! 2/3!(X)   $) {};

\coordinate [label=below:${}$] (E) at ($ (A)! 1/3!(B) $) {};
\coordinate [label=below:${}$] (C) at ($ (A)! 2/3!(B) $) {};
\coordinate [label=left:${}$] (G) at ($ (A)! 1/3!(D) $) {};
\coordinate [label=left:${}$] (H) at ($ (A)! 2/3!(D) $) {};
\coordinate  [label=right:${}$] (I) at ($ (D)! 1/3!(B) $) {};
\coordinate  [label=right:${}$] (J) at ($ (D)! 2/3!(B) $) {};

\draw[line width=2pt](A) -- (B)  -- (D) -- cycle;

\fill[gray!50](A) -- (B) -- (D) -- cycle;

\coordinate [label=right:${\quad O=\beta}$] (F) at (F);
\foreach \point in {A,B,D,F}
\fill [black] (\point) circle (3pt);

\draw  (F) circle (5pt);

\end{tikzpicture}
\caption{$\beta=(0,0,0)$ as a point of ${\mathbb A}^{\cal B}_{4,+}$}
\end{figure}
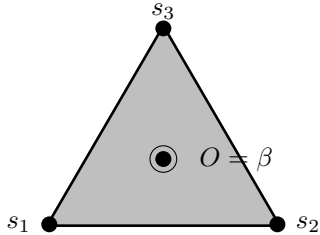
As before, we write $\beta$ as a convex combination of the points in $S$. That is, 
\begin{equation}\label{s1s2s3o}
\beta = as_1+bs_2+cs_3+pO
\end{equation}
where $a,b,c,d\in {\mathbb R}$ and $$a+b+c+p=1.$$
Substitude $\beta = O $ into (\ref{s1s2s3o}), we have
\begin{equation}\label{os1s2s3}O=\frac{1}{1-p}\, (as_1+bs_2+cs_3).\end{equation}
Since $a+b+c+p=1$, we have $a+b+c=1-p$. Thus is right hand side of (\ref{os1s2s3}) is a convex combination. But $\{s_1,s_2,s_3\}$ is an affinely independent set, so the solution of (\ref{os1s2s3}) is unique. Obviously we have
$$O=s_1+s_2+s_3,$$
thus $$
a=b=c=1-p.
$$
Let $q=1-p$, then (\ref{s1s2s3o}) becomes
\begin{equation}\label{qp}
\beta = q(s_1+s_2+s_3)+pO
\end{equation}
where $p,q\in {\mathbb R}$ and $p+q=1$.

Even though $S$ does not satisfy the condition in Corollary \ref{Fbeta}, we can still assume that $f_\beta$ is of the form 
$$\lambda_1 x^3+\lambda_2 y^3+\lambda_3 z^3+\mu_1 xyz.$$  If $f_\beta = \beta ={\rm diag}(0,0,0)$, then at least $f_\beta$ is diagonal, so Lemma \ref{diagfbeta} still works. If we force us to use Corollary \ref{Fbeta}, then (\ref{qp}) implies that 
\begin{equation}
\lambda_1=\lambda_2=\lambda_3 = \sqrt{q}, \mbox{\quad and\quad} \mu_1=\sqrt{6p}.
\end{equation}
Thus, let 
\begin{equation}\label{fbetalm}
f_\beta= \lambda(x^3+y^3+z^3)+\mu xyz,
\end{equation}
then we can check that the moment matrix ${\frak m}(f_\beta)$ of (\ref{fbetalm}) is ${\rm diag}(0,0,0)$.

For the situation of Corollary \ref{Fbeta} which is discussed for most part of this paper, the solution is unique for each $S$ if it exists. However, the family (\ref{fbetalm}) tell us what may happen for affinely dependent $S$, there may exists a family, not necessarilly one dimensional, of critical points of the function $\|\frak{m}\|^2$. 

In fact (\ref{fbetalm}) is a generic family of plane cubics called the Hesse's
canonical equations of cubic curves (see \cite{CAG}), and is dicussed in \cite{Ness}.

\section{The SAGE  Notebook}

\begin{verbatim}
sage: r=2
sage: e=RootSystem(['A',r]);e
sage: E=e.ambient_space();E
sage: WCR=WeylCharacterRing(CartanType(e));WCR
sage: fw=WCR.fundamental_weights();fw
sage: def WeightsOfHighestWeight(f): return WCR(f).weight_multiplicities()
sage: WeightsOfHighestWeight(3*fw[1])
sage: WT=[i for i in WeightsOfHighestWeight(3*fw[1])];WT
sage: wt=[[j[i] for i in range(r+1)] for j in WT];wt
sage: def LaterThan(A,C): return A[A.index(C[-1:][0][-1:][0])+1:]
sage: def Pairs(B,A): return [[[j][0]+[i] for i in LaterThan(A,[j])] for j in B]
sage: def Flat(A): return [i[j] for i in A for j in range(len(i))]
sage: def FlatP(B,A): return Flat(Pairs(B,A))
sage: def Card(A,k): 
...       B=[[j] for j in A] 
...       for i in range(k-1):
...            B=FlatP(B,A)
...       else:   return(B)
sage: def Dele(A,i): return A[:i]+A[i+1:]
sage: def Aff(A,k) : return [[Dele(A,k)[j][i]-A[k][i] for i in range(len(A[0]))]
...    for j in range(len(A)-1)]
sage: def AffRowLinearIndependence(A):
...       return matrix(Aff(A,0)).row_space().dimension()==matrix(Aff(A,0)).nrows()
sage: def AA(B,i): return transpose(matrix(Aff(B,i)))
sage: def PP(A) : 
...       return identity_matrix(A.nrows())-A*(transpose(A)*A)^(-1)*transpose(A)
sage: def NT(B): return PP(AA(B,0))*transpose(matrix(B[0]))
sage: def Inde(k): 
...       return [i for i in Card(wt,k) if AffRowLinearIndependence(i)==True]
sage: def ColMatList(j) : return [i[0] for i in j]
sage: def MatCol(j) : return matrix([[i] for i in j])
sage: def InWeylChamber(B): 
...       return max([B[1:][i]-B[:-1][i] for i in range(len(B)-1)])<=0
sage: def wts(m): 
...       return [vector([j[k] for k in range(r+1)]) for j in 
...       [i for i in WeightsOfHighestWeight(m*(WCR.fundamental_weights()[1]))]]
sage: def Diff(a,b):
...       if any(a[1][i]<b[1][i] for i in range(len(a[1])))==True or a[0]*b[0]==0:
...          return [0,vector(0 for i in range(len(a[1])))]
...       else:
...           return [a[0]*b[0]*prod(falling_factorial(a[1][i],b[1][i])
...           for i in range(len(a[1]))),
...           vector(a[1][i]-b[1][i] for i in range(len(a[1])))]
sage: def Pol(B,A):
...       if sum(B[0][1])>=sum(A[0][1]):
...          PP=[i for i in [Diff(B[k],A[j]) for j in range(len(A)) 
...             for k in range(len(B))]]
...          WTS=wts(max([sum(i[1]) for i in PP]))
...          L=len(WTS)
...          return [[sum([i[0] for i in PP if i[1]==WTS[j]]),WTS[j]]
...           for j in range(L)]
...       else:
...           return [[0,vector([0 for i in range(len(A[0][1]))])]]
sage: def Polar(B,A): 
...       PL=Pol(B,A)
...       if all(i[0]==0 for i in PL)==True:
...           return [[0,vector([0 for i in range(len(A[0][1]))])]]
...       else:
...           return [i for i in PL if i[0]!=0]
sage: wts1=[[1,wts(1)[i]] for i in range(len(wts(1)))];wts1
sage: p=sum(wt[0])/(r+1);p
sage: def Moment(A): 
...       return matrix(
...         [[Polar(Polar(A,[wts1[i]]),Polar(A,[wts1[j]]))[0][0]
...          for i in range(len(wts1))]
...          for j in range(len(wts1))])/Polar(A,A)[0][0]-p*identity_matrix(r+1)
sage: MTW=[Moment([[1,t]]) for t in wt]
sage: mtw=[[MTW[j][i][i] for i in range(r+1)] for j in range(len(MTW))]
sage: Ind=[i for i in Card(mtw,2) if AffRowLinearIndependence(i)==True]
sage: MCC=[ColMatList(j) for j in [NT(i) for i in Ind]]
sage: MMC=[[Ind[i],MCC[i]]
...       for i in range(len(Ind))
...        if Polyhedron(vertices=Ind[i]).contains(MCC[i])==True]
sage: MMW=[i for i in MMC if InWeylChamber(i[1])==True];MMW
sage: [i for i in MMW if i[1]!=[0 for j in range(r+1)]]
\end{verbatim}

%\begin{acknowledgements}
%\end{acknowledgements}


\begin{thebibliography}{}

\bibitem{Atiyah} Atiyah, M.F.: Convexity and commuting Hamiltonians, Bull. London
Math. Soc. 14(1),1-15 (1982)

\bibitem{CLS} Cox, D., Little, J., Schenck, H.: Toric Varieties,  American Mathematical Society,Providence, R.I. (2011)

\bibitem{Derksen-Kemper} Derksen, H., Kemper, G.:  Computational Invariant Theory, Springer (2002)

\bibitem{CAG}Dolgachev, I: Classical algebraic geometry : a modern view, Cambridge University Press, London (2012)

\bibitem{DolInv} Dolgachev, I: Lectures on invariant theory, Cambridge University Press ,Cambridge (2003)

\bibitem{Dolgachev-Hu} Dolgachev, I., Hu,Y.: 
Variation of geometric invariant theory quotients, Publications Mathematiques de l'Institut des Hautes 
{\'E}?tudes Scientifiques.87(1), 5-51, (1998)

\bibitem{GTM168} Ewald, G.:  Combinatorial Convexity and Algebraic Geometry, Graduate Texts in Mathematics, 168, Springer (1996)


\bibitem{Stabilityg5}Fedorchuk, M., Smyth, D.I.: Stability of genus five canonical curves, A Celebration of Algebraic Geometry,  American Mathematical Society (2013)

\bibitem{Hesselink} Hesselink, W.: Desingularization of varieties of null forms, Invent. Math. 55, 141-163 (1979)


\bibitem{Huang} Huang, H.X.,Liang, Z.A.,Pardalos, P.M.: Some properties for the Euclidean distance matrix and 
positive semidefinite matrix completion problems, J. of Global Optim. 25, 3-21 (2003)

\bibitem{Kempf}  Kempf, G.,  Ness, L.: The length of vectors in representation spaces, in ???Algebraic
geometry, Copenhagen 1978???, Lecture Notes in Math. 732, 233-243 Springer-Verlag (1979)

\bibitem{Kirwan}Kirwan, F.:  Cohomology of quotients in symplectic and algebraic geometry, Princeton
University Press (1984)

\bibitem{Kirwan1}Kirwan, F: Refinements of the Morse stratification of the normsquare of the moment map in Mathematics,  232, 327-362 (2005)

\bibitem{MW} Marsden, J., Weinstein, A.: Reduction of symplectic manifolds with symmetry. Reports on
Math. Phys. 5, 121-130 (1974)

\bibitem{IMorrison}Morrison, I.: GIT constructions of moduli spaces of stable curves and maps, Surv. Differ. Geom., 14, 315-369 (2009)


\bibitem{MumStability}Mumford, D.: Stability of projective varieties, L'??Enseignement Math. 23, 39-110 (1977)



\bibitem{MFK} Mumford, D., Fogarty, J., Kirwan, F.:  Geometric Invariant Theory, 3rd ed.,  Springer-Verlag, Berlin, New York (1994)


\bibitem{Ness}  Ness, L.:  A stratification of the null cone via the moment map, Amer. Jour. of Math.
106,1281-1325 (1984)


\bibitem{LOSG}  da Silva, A.C.:  Lectures on Symplectic Geometry,  Lecture Notes in Mathematics Volume 1764, Springer (2008)


























\end{thebibliography}


\end{document}